%% file: hol3.tex
\nonstopmode

\documentclass[11pt]{article} 

\usepackage{cite}
\usepackage{amsmath}
\usepackage{amsthm}
\usepackage{amsfonts}
\usepackage{amssymb}
\usepackage{mathtools}
\usepackage{mathrsfs}
\usepackage{enumerate}

\newtheorem{thm}{Theorem}
\newtheorem{lem}[thm]{Lemma}
\newtheorem{prop}[thm]{Proposition}

\theoremstyle{definition}

\newtheorem{conj}[thm]{Conjecture}

\newtheorem{example}[thm]{Example}
\theoremstyle{remark}
\newtheorem{rmk}[thm]{Remark}

\DeclareMathOperator{\sgn}{sgn}

\DeclareMathOperator{\supp}{supp}

\DeclareMathOperator{\dist}{dist}
\DeclareMathOperator{\divergence}{div}

\DeclareMathOperator{\vol}{vol}

\DeclareMathOperator{\hessian}{Hess}

\DeclareMathOperator{\length}{length}

\newcommand{\C}{\mathbb{C}}
\newcommand{\R}{\mathbb{R}}
\newcommand{\Z}{\mathbb{Z}}

\newcommand{\Cinfc}{{C^\infty_c(T^nM)}}
\newcommand{\Mn}{{\mathscr M_n}}
\newcommand{\mass}{{\mathbf M}}
\newcommand{\holonomic}{\mathscr{H}}
\newcommand{\N}{\mathbb{N}}
\newcommand{\distributions}{{\mathscr{D}'(T^nM)}}
\newcommand{\distributionsP}{{\mathscr{D}'(P)}}
\newcommand{\Dn}{{\mathscr{D}'_n}}
\newcommand{\Pn}{{\mathscr{V}_n}}
\newcommand{\ddtzero}{{\left.\frac{d}{dt}\right|_{t=0}}}

\newcommand{\diffcondition}{{(Pos)}}
\newcommand{\holcondition}{{(Hol)}}
\newcommand{\critcondition}{{(Crit)}}
\newcommand{\homcondition}{{(Hom)}}
\newcommand{\probcondition}{{(Prob)}}

\newcommand{\posfuncs}{\mathscr F_\mu}

\date{}
\title{The variational structure of the space of holonomic measures}
\author{Rodolfo R\'ios-Zertuche}

\begin{document}

\maketitle

\noindent{\hspace{\fill}\emph{To Xavier G\'omez-Mont \'Avalos}\hspace{\fill}}

\input{abstract}

\tableofcontents
\section{Introduction}
\label{sec:introduction}
\input{intro}

\section{Distributions that arise as derivatives of families of measures}
\label{sec:families}

\input{families}

\section{Preliminaries on holonomic measures}
\label{sec:preliminaries}
\subsection{Setting}
\label{sec:setting}
\input{setting}
\subsection{Definition of holonomic measures and their topology}
\label{sec:measures}
\input{measures}

\section{The holonomic tangent space}
\label{sec:tangent}
 \subsection{Mild distributions}
 \label{sec:distributions}
 \input{distributions}
\subsection{Two-sided variations of holonomic measures}
 \label{sec:variations}
 \input{tangent}

\subsection{One-sided variations of holonomic measures}
\label{sec:onesidedholonomic}
\input{onesidedhol}

\section{Examples}
\label{sec:examples}
\subsection{Criticality}
\input{examples}
\subsection{Horizontal variations}
\label{sec:horizontal}
\input{horizontal}

\subsection{Vertical variations}
\label{sec:vertical}
\input{vertical}

\subsection{Transpositional variations}
\label{sec:transpositional}
\input{transpositional}

 \bibliography{bib}{}
 \bibliographystyle{plain}
\end{document}

%% file: abstract.tex
\begin{abstract}
Roughly speaking, holonomic measures are parametric varifolds without boundary. They provide a setting appropriate for the analysis of many variational problems.  In this paper, we characterize the space of variations for these objects, and we use the characterization to formulate stability conditions that are strictly more general than the Euler-Lagrange equations. We also use this characterization to deduce higher-dimensional analogues of energy conservation and weak KAM.

Along the way, we characterize the distributions that arise as derivatives of families of Borel probability measures on smooth manifolds.
\end{abstract}

%% file: intro.tex
In this paper we consider the space of holonomic measures on a manifold $M$ with tangent bundle $TM$. These are roughly speaking all Borel measures on $T^nM=TM\oplus\cdots\oplus TM$ that can be approximated by $n$-dimensional cycles. One can also say that they are parametric varifolds without boundary, because they induce a varifold and also encode a local parameterization for it. These are defined carefully in Section \ref{sec:preliminaries}. 

We study the ways in which these measures can be deformed, thus characterizing the velocity vectors of all curves in the space of holonomic measures that are differentiable in a certain sense. We are thus able to give a good description of the tangent bundle to the space of holonomic measures. We do this in Section \ref{sec:tangent}.

This study is fruitful, as is shown by an initial set of applications presented in Section \ref{sec:examples}. Among other things, we are able to show that the conditions we obtain for criticality are effectively more general than the classical Euler-Lagrange equations. We also show analogues of energy conservation and of the weak KAM theorem, for which we are missing a regularity result however.

In order to achieve our goal, we prove a general characterization of all distributions that appear as derivatives of both signed, positive, and probability measures. We do this in Section \ref{sec:families}.

In order to clarify the general ideas and goals of the paper, we present in Section \ref{sec:toyexample} an example that illustrates the general philosophy, and in Section \ref{sec:examplesoflagrangians} a series of examples of different Lagrangians that could be studied using holonomic measures. We then present a brief survey of related literature in Section \ref{sec:literature}. In Section \ref{sec:acknowledgements}, we acknowledge the people and institutions that contributed to this work.

\subsection{Introductory example: length-minimizing curves have no corners}
\label{sec:toyexample}

Let $M$ be the flat two-dimensional torus $M=\R^2/\Z^2$.
Let $\gamma:[0,2)\to M$ be a closed curve, given by 
\[\gamma(s)=\left\{\begin{array}{ll}
(s,0)\mod \Z^2,&s\in[0,1), \\
(0,s)\mod \Z^2,&s\in[1,2).
\end{array}\right.\]
The curve $\gamma$ induces a probability measure $\mu_\gamma$ on $TM$ by pushing forward the uniform probability on $[0,2)$ by use of the map $d\gamma\colon t\mapsto (\gamma(t),\gamma'(t))$. Thus, for measurable $f\colon TM\to\R$, we have
\begin{multline*}
\int f\,d\mu_\gamma=\frac12\int_0^2f(\gamma(t),\gamma'(t))\,dt\\
=\frac12\int_0^1f((s,0),(1,0))\,dt+\frac12\int_1^2f((0,s),(0,1))\,dt.
\end{multline*}
The support of $\mu_\gamma$ is exactly the set of velocity vectors of $\gamma$. The measure $\mu_\gamma$ is an example of what we call a holonomic measure. It encodes a 1-dimensional submanifold with a singularity at $(0,0)$, together with its parameterization. Note that at the origin $(0,0)$ the two components of $\gamma$ cross, so that there are two tangent vectors:
\begin{equation}\label{eq:tangentdescription}
\supp\mu_\gamma\cap T_{(0,0)}M=\{(1,0),(0,1)\}.
\end{equation}

In the traditional proof that our curve $\gamma$ is not a geodesic, we would proceed by comparing its length with the length of a curve that takes ``shortcuts'' near $(0,0)$ to cut the corner. Such a comparison proof is easy to complete in the simple case of a minimizer of 1-dimensional length, but it could be more difficult to produce such a construction in higher dimensions and for any smooth Lagrangian $L$.

We are interested in extracting as much information about the minimizers as can be obtained from doing variations of them. For example, we know intuitively that there is some variation of our curve $\gamma$ above that looks like the one in Figure \ref{fig:variation}. Let us call $\gamma_t$ the curve in this variation corresponding to each time $t\in \R$. In this variation, at $t=0$ we have the original curve $\gamma$. For negative $t$, we have a curve of length larger than that of $\gamma$. For positive $t$, the curves have length smaller than that of $\gamma$. We thus have
\[\ddtzero\length(\gamma_t)<0,\]
which immediately proves that $\gamma$ cannot be a minimizer of the length.

\begin{figure}
  \centering
    \includegraphics[width=0.9\textwidth]{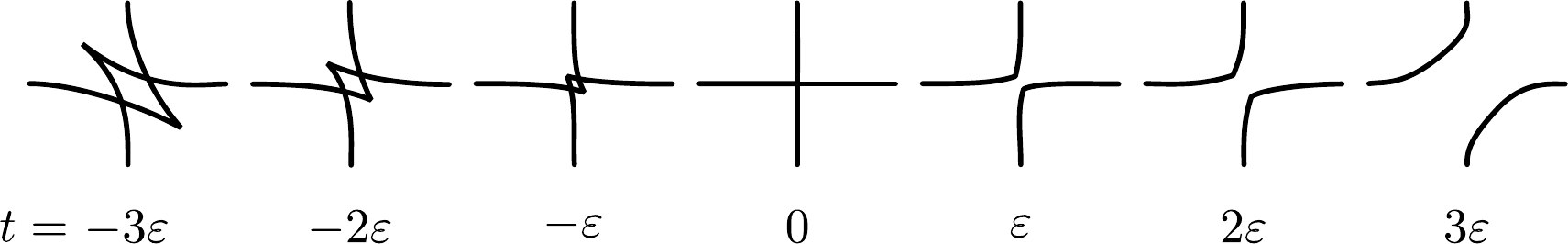}
  \caption{A variation of $\gamma$ near $(0,0)$ at different times $t$ for some $0<\varepsilon\ll1$.}
  \label{fig:variation}
\end{figure}

Thus if we knew that such a variation existed every time the support of the measure $\mu_\gamma$ intersected a fiber of $TM$ in two points, we would be able to conclude that minimizers cannot have corners.

This is all intuitively true, but the variation of Figure \ref{fig:variation} is rather hard to write down. Instead of writing it down, we write down the desired derivative $d\gamma_t/dt|_{t=0}$ of the family of curves $\gamma_t$, and we compute the derivative $\ddtzero\length(\gamma_t)$. Theorem \ref{thm:tangent} will guarantee the existence of some variation with that derivative, like the one in Figure \ref{fig:variation} --- which we no longer need to construct explicitly. 

Instead of trying to write down $\ddtzero\gamma_t$, we consider the family of measures $\mu_t$ induced by the curves $\gamma_t$, and we write down the derivative $\ddtzero\mu_t$, which will be a distribution $\eta$.  Perhaps the main property of the family $\gamma_t$ at 0 is that all the movement is (infinitesimally) happening precisely in the tangent space to $(0,0)\in M$, and the main component of the movement is really the change in the directions of the vectors tangent to the curve, which are infinitesimally getting closer as $t$ increases, as illustrated in Figure \ref{fig:tangents}. While the curve does move away from $(0,0)$, we can actually assume that that effect occurs as $O(t^2)$, and is hence not important for the derivative $\eta$. Thus we only need $\eta$ to reflect the movement of the tangent vectors, and it should look like the scheme in Figure \ref{fig:derivative}. 

\begin{figure}
  \centering
    \includegraphics[width=0.9\textwidth]{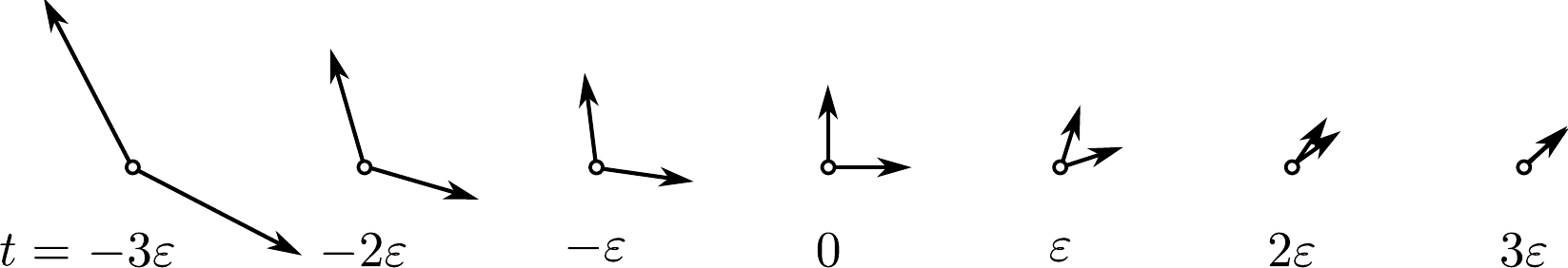}
  \caption{Evolution of the tangent vectors to the curve $\gamma$ of Figure \ref{fig:variation} at the singular point.}
  \label{fig:tangents}
\end{figure}

\begin{figure}
  \centering
    \includegraphics[width=0.08\textwidth]{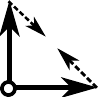}
  \caption{Schematic representation of the derivative $\eta=\ddtzero\mu_t$, deduced from Figure \ref{fig:tangents}.}
  \label{fig:derivative}
\end{figure}

We need to make an interlude to explain how to represent movement with a distribution. The easiest example is to take the family of Dirac deltas $\delta_t$ on $\R$, which move with $t$. The derivative of the movement at 0 should be the tangent vector $v=(1)$ to $\R$. As a distribution, we see that 
\[\ddtzero \langle \delta_t,f\rangle=\ddtzero f(t)=f'(0)=\langle-\partial \delta_0,f\rangle.\]
In other words, the derivative $\ddtzero\delta_t=-\partial \delta_0$ is the distributional way to express movement from left to right on the real line, at $x=0$.

To implement the movement schematized in Figure \ref{fig:derivative}, we first see that the measure $\mu_0=\mu_\gamma$, when restricted to $T_{(0,0)}M$, is really a sum of two deltas at the points of its support (compare with equation \eqref{eq:tangentdescription}). We want to move the one at $(1,0)$ in the direction $(-1,1)$, so we will have a derivative $-\partial_{(-1,1)}\delta_{((0,0),(1,0))}$. Similarly, the one at $(1,0)$ should move in the direction $(1,-1)$, which gives a component $-\partial_{(1,-1)}\delta_{((0,0),(0,1))}$. (The derivatives act in the direction of the fibers of the tangent bundle.)  In other words, we have the distribution $\eta$ on $TM$ given by
\[\eta=-\partial_{(-1,1)}\delta_{((0,0),(1,0))} -\partial_{(1,-1)}\delta_{((0,0),(0,1))}.\]

The conditions of Theorem \ref{thm:tangent} are easy to check. By Lemma \ref{lem:Poseasy}, Condition \diffcondition\ holds. Since $\langle\eta,1\rangle=0$, Condition \probcondition\ holds. To see that Condition \holcondition\ holds, we pick $f\in C^\infty(M)$ and compute
\begin{align*}
\langle \eta, df\rangle&=(f_x,f_y)\cdot((-1,1)+(1,-1))=0
\end{align*}
so the condition is satisfied.
This means that $\eta$ is indeed the derivative of some family $\mu_t$ of holonomic measures.

Now, let us use this to check that $\gamma$ is indeed not a geodesic. We let $L(x,v)=\sqrt{v\cdot v}$, so that the integral 
\[\int L\,d\mu_\gamma=\int_0^2 |\gamma'(s)|\,ds=\length(\gamma)\]
expresses the length functional as a Lagrangian action with Lagrangian density $L$.
If $\gamma$ were a critical point of this length functional, the derivative
\[\ddtzero\length(\gamma_t)=\ddtzero\int L\,d\mu_t=\langle\eta,L\rangle\]
would vanish for all variations $\gamma_t$. Thus, since
\begin{multline*}
\langle\eta,L\rangle=\nabla|v|\cdot(-1,1)|_{v=(1,0)}+\nabla|v|\cdot(1,-1)|_{v=(0,1)} \\
=\left.\frac{-v_1+v_2}{|v|^3}\right|_{v=(1,0)}+\left.\frac{v_1-v_2}{|v|^3}\right|_{v=(0,1)}=-1-1=-2\neq 0,
\end{multline*}
we conclude that $\gamma$ is not a critical point of the length, hence also not a minimizer.
Which is what we expected from the intuition given to us from Figure \ref{fig:variation}.

In this particular case we are, of course, reaching for the sledgehammer to crack a nut. However, the examples of Section \ref{sec:examples} will show that it is often easy and worthwhile to consider large families of variations and to extract infteresting information from them.

\subsection{Examples of Lagrangians}\label{sec:examplesoflagrangians}

Geometric measure theory has traditionally worried mostly about problems in which the Lagrangian has lots of symmetries. Good examples are the $k$-dimensional area (or mass), that is, when the Lagrangian is the volume $\vol_k$ induced by a Riemannian metric $g$ on the manifold $M$, by
\[L(x,v_1,\dots,v_n)=\vol_k(v_1,\dots,v_k)=\left|\det(g(v_i,v_j))_{i,j=1}^k\right|\]
and the classic mechanical Lagrangians, of the form
\[L(x,v)=\frac12{g(v,v)}+V(x),\]
for some potential $V$.
We want to argue that there are many interesting situations in which one cannot expect so much symmetry.

\begin{example}
Consider the case of socks, which are generally made of textile tissue with the property that it is more stretchy in one direction than in the other. In this case, an appropriate model would probably be of the form
\[L(x,v_1,v_2)=\vol_2(v_1,\beta v_2)+V(x)\]
where $0<\beta\neq 1$ is a parameter that will account for the difference in stretchiness, and $V$ is some potential (intended to keep the sock away from the foot). Notice that socks usually have seams, and the objects we propose, holonomic measures, allow for this.
\end{example}

\begin{example}
The construction of a flexible computer screen out of organic electronics to wrap a geometric body, such as could be used decoratively on an innovative architectural piece, would probably have different costs depending on the direction in which it were laid out, thus becoming the solution of an anisotropic Lagrangian optimization problem.
\end{example}

\begin{example}
A plant's stem can be thought of as a 3-dimensional cell minimizing an energy that is different in the longitudinal direction than in the radial direction, and for which the source of light matters. It is also likely that a position-dependent potential would have to be added to account for the problems that height brings, like difficulty in water transportation. Hence, the appropriate Lagrangian would be anisotropic and would have symmetry neither with respect to $\mathrm{GL}(d,\R)$ (acting on $TM$) because the direction of the light matters, nor with respect to $S_3$ (acting on the indices of $v_1,v_2,v_3$) because longitudinal and radial growth have different costs.
\end{example}

\begin{example}
Another example is that of laminations that locally look like harmonic from a given manifold. In this case, the Lagrangian looks like an anisotropic
quadratic form on the partial derivatives of the local parameterization. 

While the existence of harmonic maps has been proved in some cases ---notably in the case in which the target manifold has nonnegative sectional curvature \cite{eellssampson}---, it is known that there are no harmonic maps in many other cases. For instance, there are no harmonic maps from the sphere $S^2$ into itself (see for example \cite{jost}). There is a body of literature devoted to suggesting alternatives to harmonic maps in the cases in which those do not exist. Using holonomic measures, one always gets an energy minimizer that may not correspond to an immersed manifold.  It would be interesting to understand these minimizers more deeply.
\end{example}

\subsection{Related literature.}\label{sec:literature}
Geometric measure theory and variational analysis are vast subjects, so a discussion about how this research fits in those contexts is in place. However, since it seems impossible to give an exhaustive discussion, we choose to instead give just a brief one and hence minimize the number of mistakes we make in the process. Also, we will not define all the objects involved, but rather we will just mention them in the hope that readers familiar with these concepts will find the information they are looking for, while readers not familiar with them will be happy to ignore the discussion.

Throughout this paper, $d\geq1$ will denote the dimension of the ambient mainfold $M$, while $1\leq n\leq d$ will denote the dimension of the holonomic measures. This roughly means that we are considering submanifolds of dimension $n$.

Holonomic measures appeared in the $n=1$ case in Mather's \cite{matheractionminimizing91} version of Mather-Aubry theory for minimizers of the action of Lagrangians on the torus. The theory of holonomic measures was extended by others; for example by Ma\~n\'e \cite{manhe,contrerasiturriagabook}, Bangert \cite{bangert}, Bernard \cite{patrick}. A certain case of codimension one of Mather-Aubry theory was considered by Moser \cite{moser1986,moser1987,moser1988}.

In the more general context we treat here, in which $n$ can be arbitrary, a similar theory should exist for a large class of Lagrangians. Under rather mild conditions in the Lagrangian (such as convexity, coercivity superlinearity, tightness, quasiconvexity) minimizers exist in all holonomy classes with coefficients in the real numbers $\R$. However, analogues of Mather's $\alpha$ and $\beta$ functions are probably only defined for a very restricted set of Lagrangians.

Holonomic measures induce superpositions of \emph{currents} (cf. \cite{federer,morgan}) on a manifold $M$ in an obvious way. However, they carry more information than currents because they take into account the parameterization and orientation of the minimizers, and hence allow for the study of anisotropic Lagrangians. 

Holonomic measures also induce \emph{varifolds} (cf. \cite{almgrenvarifolds,allard,simon}). Again, they carry more information because they record not only the tangent planes, but also the velocity vectors of a `parameterization,' and the orientation. Our characterization of the tangent bundle to the space of holonomic measures also implies a characterization of the space of first variations of a varifold, which turns out to be larger than the set considered traditionally since the work of Allard \cite{allard}; for a summary of that theory see for example \cite{leonardimasnou}. 

Similarly, holonomic measures also have more structure than \emph{sets of finite perimeter} (see for example \cite{maggi}). Although in this paper we consider only objects without boundary (i.e., with empty perimeter), it is easy to use the variations of Proposition \ref{prop:holonomic} of the present paper that are given in \cite{myclosedmeasuresareholonomic} to get similar results to those explained here in the case in which the sets have boundary, and in those cases one should indeed require certain finiteness conditions.

The set of holonomic measures contains representations of the families of the \emph{cone and cup competitors} and the \emph{deformed competitors} for the direct approaches to Plateau's problem discussed in the recent papers \cite{delellisghiraldinmaggi,dephillipsderosaghiraldin}. Again, those do not carry information about the parameterization or the orientation of the minimizers. Similar remarks correspond to the geometric approaches of \cite{reifenberg60,reifenberg64a,reifenberg64b,feuvrier09,depauw09,harrison11,harrison14,harrisonpugh,fang13}. The disadvantage of holonomic measures with respect to those lies in the lack of clear geometric structure (i.e., our set of `competitors' is much larger and harder to describe \emph{a priori}). A good review of other alternatives is given in \cite{david14}.

With holonomic measures the issue of rectifiability is not a concern since rectifiability is built into them. Whether or not one can find their volume (or the action of a Lagrangian) depends on the question of whether this function is integrable with respect to them.

Holonomic measures are suitable for the treatment of many problems that could be approached parametrically using functions for example in Sobolev or Lipschitz spaces (cf. \cite{evans,dacorogna,kristensenmingione, giusti}). 

Superpositions of \emph{Young measures} (cf. \cite{patrick,young}) are a special case of holonomic measures. 

In Section \ref{sec:vertical} we deduce a sort of general Hamilton-Jacobi equation, a case of which has been studied to great depth (see for example \cite{crandalllions,fathibook}).

The definition of differentiability of families of measures (i.e., of varitions) that we use is only one possibility of many; see for example \cite{smolyanov} for an exploration of other possibilities.

\subsection{Acknowledgements.}\label{sec:acknowledgements}
I am deeply indebted to John N. Mather for his patience in listening to a number of sometimes very confused and tentative presentations of these results and for his help in clarifying my ideas with numerous questions and suggestions. I am also very grateful to Antonio Ache, Camilo Arias Abad, Victor Bangert, Patrick Bernard, Albert Fathi, Jes\'us Puente Arrubarrena, and Stefan Suhr for several  conversations on this subject, to Gonzalo Contreras and Renato Iturriaga for introducing me to the calculus of variations, and to Burglind Juhl-J\"oricke for teaching me what I know about distributions. I am also very grateful to Luigi Ambrosio for very helpful comments.

I am very grateful to Princeton University, to the Institute for Computational and Experimental Research in Mathematics at Brown University, and to the Max Planck Institute for Mathematics in Bonn for their hospitality and support during the development of this research.

%% file: families.tex
Throughout this section, let $P$ be a $C^\infty$ manifold of dimension $m$ without boundary, and let $\mu$ be a Borel measure on $P$. Denote by $C^\infty_c(P)$ the space of smooth functions with compact support on $P$. 
%

In this section, 
characterize the velocity vectors for curves in the space of Borel measures on $P$ that pass through $\mu$.  These velocity vectors are given by certain distributions.

We find that if the measures $\mu_s$ are allowed to be signed (i.e., to have both positive and negative mass), then any distribution can arise; see Proposition \ref{prop:characterizationsigned}. On the other hand, if the measures $\mu_s$ are only allowed to be positive, we find a necessary and sufficient condition for a given distribution to be the velocity vector of a curve through $\mu_0$. This is Condition \diffcondition\ below, which says that the nullspace of the distribution must contain all smooth, nonnegative functions that vanish on the support of $\mu_0$. This characterization is our main result of the section, and it is given in Theorem \ref{thm:generalfamilies}. This theorem also accounts for the case in which all the measures $\mu_s$ are probabilities. We also look at the case of one-sided derivatives in Section \ref{sec:onesided}.

In Section \ref{sec:flows} we explain how one can use Colombeau algebras to build the bridge with the classical ideas of mass transport.

Interest in the variational structure of the space of measures, which we study here, comes from the applications that the analysis of measures has found for example in problems of optimal transport (e.g., \cite{ambrosio,ambrosiogiglisavare}) and optimization, as in Mather-Aubry theory (e.g., \cite{matheractionminimizing91,contrerasiturriagabook}). Differentiable families of measures have also been studied extensively for example in \cite{smolyanov}. Our own applications appear in Section \ref{sec:tangent}. 

The variational structure of the space of measures has been explored, with a stronger topology that results in a smaller tangent space, in \cite[Sections 8.4 and 8.5]{ambrosiogiglisavare}. 

We give precise definitions and some preliminaries in Section \ref{sec:subdistributions}, and we state and prove our result in Section \ref{sec:general_deformations}.
In Section \ref{sec:flows} we give some comments regarding what these results mean for mass transport and flows.

\subsection{Distributions and measures}\label{sec:subdistributions}
\subsubsection{Convolutions}\label{sec:smoothing}

A \emph{mollifier} is a function $\psi\in C_c^\infty(\R)$ such that $\psi(x)=\psi(-x)$, $\int\psi=1$, and $\psi\geq 0$. 

We will say that a tuple of vector fields $F=(F_1,\dots,F_\ell)$ on $P$ is \emph{generating} if at every point $p\in P$ the vectors $F_1(p),\dots,F_\ell(p)$ span all of the tangent space $T_pP$.

Fix a generating tuple of vector fields $F=(F_1,\dots,F_\ell)$. Denote by $\phi^i:P\times \R\to P$ the \emph{flow} of $F_i$:
\[\phi^i_0(x)=0,\quad\frac{d\phi^i_s(x)}{ds}=F_i(\phi^i_s(x)),\quad s\in \R.\]
For $f\in C^\infty_c(P)$, we will denote by $P_i(f)$ the function given by
\[P_i(f)(x)=\int_\R f\circ \phi_s^i(x)\,\psi(s)\,ds.\]
This is a \emph{convolution} in the direction $F_i$.

For $f\in C^\infty_c(P)$, we will denote
\[\psi*_Ff:=P_1P_2\cdots P_\ell (f).\]

\subsubsection{Definition and smoothing of distributions}
\label{sec:distsmoothing}

A \emph{distribution} on the open set $U\subseteq\R^m$ is a linear functional $\eta:C^\infty_c(U)\to\R$ such that for each compact set $K\subset U$ there are some constants $N>0$ and $C>0$ (depending only on $K$ and $\eta$) such that
\[|\langle \eta,f\rangle|\leq C\sum_{|I|\leq N}\sup_{p\in U} |\partial^I f(p)| \]
for all $f\in C^\infty_c(U)$.
Here, the sum is taken over all multi-indices $I$ with $m$ nonnegative entries adding up to at most $N$, and $\partial^I$ denotes the iterated partial derivatives in the corresponding directions in $\R^m$.

We fix, once and for all, an $n$-dimensional $C^\infty$ manifold $P$ without boundary, and with a Riemannian metric that induces the distance $\dist_P$ between points of $P$. 

Let $\eta\colon C^\infty_c(P)\to\R$ be a linear functional.
For a chart $\varepsilon\colon U\to W$ from the open set $U\subseteq P$ to the open set $W\subseteq \R^n$, the \emph{pushforward} $\varepsilon_*\eta$ is defined by
\[\langle \varepsilon_*\eta,f\rangle=\langle\eta,f\circ\varepsilon\rangle\]
for $f$ in $C^\infty_c(W)$.

The functional $\eta$ is a \emph{distribution} if for each chart $\varepsilon$ as above, $\varepsilon_*\eta$ is a distribution on $W$.
We will denote by $\distributionsP$ the space of distributions on $P$. 
The topology on $\distributionsP$ is induced by the seminorms 
\[\eta\mapsto|\langle \eta , f \rangle|\]
for $f\in C^\infty_c(P)$. In other words, we have $\eta_i\to\eta$ if, and only if, $\langle\eta_i,f\rangle\to\langle\eta,f\rangle$ for all $f\in C^\infty_c(P)$. We remark that any measure on $P$ determines a distribution, but that not all distributions arise in this way.

For a distribution $\eta\in \distributionsP$, we define the convolution by duality:
\[\langle \psi*_F\eta,f\rangle=\langle\eta,\psi*_Ff\rangle.\]

\begin{lem}
\label{lem:measuredistribution}
 If $\eta$ is a distribution in $\distributionsP$, $F$ is a generating tuple of vector fields, and $\psi$ is a mollifier, then $\psi*_F\eta$ is a smooth signed Borel measure. 
\end{lem}
 For a proof see for example \cite[\S5.2]{friedlander}.

\subsubsection{Structure}
\label{sec:diststructure}

We fix a generating tuple $F=(F_1,\dots,F_\ell)$ of vector fields.
As before, we denote by $I$ a multi-index $I=(i_1,\dots,i_\ell)$ with $\ell$ nonnegative entries, and by $\partial^I$ the operator that iteratively takes $i_j$ covariant derivatives in the direction $F_{j}$, $j=1,\dots,\ell$.

 As usual in the theory of distributions, we define derivatives of distributions $\nu$ by duality,
 \[\langle\partial^I\nu,f\rangle=(-1)^{|I|}\langle\nu,\partial^If\rangle,\]
 and the \emph{support} $\supp\nu$ of a distribution $\nu$ to be largest set such that if $f\in C^\infty_c(P)$ is supported outside $\supp\nu$ then $\langle\nu,f\rangle=0$.

\begin{lem}[Structural representation in terms of measures]\label{lem:decompositionmeasure}
A distribution $\eta\in\distributionsP$ can be written as a sum
\begin{equation}\label{eq:decompositionmeasure}
\eta=\sum_I \partial^I\nu_I 
\end{equation}
where $I$ ranges over all multi-indices as above; for each $I$, $\nu_I$ is a signed measure. For a compact set $K\subseteq P$, 
\[K\cap\supp\nu_I=\varnothing\]
for all but finitely many multi-indices $I$.
\end{lem}

\begin{proof}
Take a partition of unity $\{\xi_j\}_{j\in \N}\subseteq C^\infty_c(P)$ of $P$, that is, a countable set of smooth functions $\xi_i$ with compact support such that $\sum_j\xi_j(p)=1$ and, on each compact set $K\subseteq P$, the restriction $\xi_j|_K\equiv 0$ for all but finitely many $j\in\N$.  We make the further assumption that the support of each of the functions $\xi_j$ is contained in an open set $U_j\subseteq P$ that is diffeomorphic to a cube $(0,1)^{n}$, and we let $\phi_j:U_j\to (0,1)^{n}$ be the corresponding diffeomorphism. 

We let $\tilde\eta_j$ be the distribution on $\R^n$ that results from pushing $\xi_j\eta$ forward to the cube $(0,1)^{n}$ and extending periodically. In other words, for all rapidly-decreasing (Schwartz) functions $f\in C^\infty(\R^{n})$, we let $\tau_zf(x)=f(x-z)$ and
\[
\langle \tilde\eta_j,f\rangle= \sum_{z\in\Z^{n}}\langle \eta,\xi_j\cdot (\tau_zf)\circ\phi_j\rangle.
\]
Like all periodic distributions, $\tilde\eta_j$ is a tempered distribution.
We have
\begin{lem}\label{lem:structuretempered}
 Every tempered distribution is a derivative of finite order of some continuous function of polynomial growth.
\end{lem}

For a proof, see for example {\cite[Theorem 3.8.1]{friedlander}}.

Let $\zeta_j$ be the continuous function of polynomial growth corresponding to $\tilde\eta_j$ (as furnished by Lemma \ref{lem:structuretempered}) and let $I_j$ be the multi-index corresponding to the derivative in the lemma, so that
\[\tilde\eta_j=\partial^{I_j}\zeta_j.\]
Let $D_j$ be the (smooth) differential operator on $P$ such that $\phi_j^*\partial^{I_j} = D_j\phi_j^*$, where $\phi^*_j$ denotes the pullback by $\phi_j$. Observe that
\[\xi_j\eta=\phi_j^*\partial^{I_j}\zeta_j=D_j\phi^*_j\zeta_j.\]
Since $\zeta_j$ is a continuous function, $\phi^*_j\zeta_j$ is piecewise continuous, and hence it induces a measure on $U_j$.
Then we can write
\[\eta=\sum_j\xi_j\eta=\sum_j D_j\phi_j^*\zeta_j,\]
and since each of the summands on the right can be expressed as a finite sum of derivatives of a continuous function, this proves the lemma.
\end{proof}

\subsection{Variations}
\label{sec:general_deformations}

Let $\mu_s$ be a family of Borel measures on the manifold $P$ parameterized by a real parameter $s$ with values in an open interval $J\subseteq\R$ that contains 0. We say that the family $\mu_s$ is \emph{differentiable} at $s=0$ if there is a distribution $\eta\in\distributionsP$ such that, for every function $f\in C^\infty_c(P)$,
\begin{equation}\label{eq:derivativefamily}
\left.\frac{d}{ds}\right|_{s=0}\int f\,d\mu_s=\langle\eta,f\rangle.
\end{equation}
The distribution $\eta$ is the \emph{derivative} $d\mu_s/ds|_{s=0}$ of $\mu_s$ at $s=0$. 

If the limit \eqref{eq:derivativefamily} exists only when restricting to $s\geq 0$, we say that the family $\mu_s$ is \emph{differentiable on one side}, and that the distribution $\eta$ is the \emph{one-sided derivative} of $\mu_s$ at $s=0$. 

We first consider the case of two-sided derivatives in Section \ref{sec:twosided}, and then in Section \ref{sec:onesided} we explain what happens for the case of one-sided derivatives.

\begin{rmk}
This is just one way to define differentiability of families of distributions; other ways have been explored for example in \cite{smolyanov}.
\end{rmk}

\subsubsection{Two-sided derivatives}
\label{sec:twosided}
If we do not restrict to the case of positive measures, we get the following result. 

\begin{prop}\label{prop:characterizationsigned}
  For every Borel measure $\mu$ and every distribution $\eta$ on $P$, there exists a family $(\mu_s)_s$ of signed Borel measures with 
  \[\mu_0=\mu\quad\textrm{and}\quad\left.\frac{d\mu_s}{ds}\right|_{s=0}=\eta.\]
\end{prop}
We prove this below. For the proof, we need to define a family of distributions $(\eta_s)_s$ to be \emph{differentiable} if there is a distribution $\nu$ such that for every function $f\in C^\infty_c(P)$, 
\[\left.\frac{d}{ds}\langle \eta_s,f\rangle\right|_{s=0}=\langle \nu,f\rangle.\]
\begin{lem}\label{lem:convolutionderivative}
For a generating tuple $F$ of vector fields, a mollifier $\psi$, and any family of distributions $(\eta_s)_s$ differentiable at $s=0$, we have
\[\left.\frac{d}{ds}\right|_{s=0}\psi*_{s^2F}\eta_s=\left.\frac{d\eta_s}{ds}\right|_{s=0}.\]
\end{lem}
\begin{proof}
This follows from the fact that for $f\in C^\infty_c(P)$, $s\mapsto\psi*_{s^2F} f$ is an even function, so its derivative at $s=0$ must vanish.
\end{proof}
\begin{proof}[Proof of Proposition \ref{prop:characterizationsigned}]
Take a mollifier $\psi$ and a generating tuple $F$ of vector fields. Then, as follows from Lemmas \ref{lem:measuredistribution} and \ref{lem:convolutionderivative}, the family $\mu_s=\psi*_{s^2F}(\mu_0+s\eta)$ has the required properties.
\end{proof}

For families of \emph{positive} measures, the situation is different.

 \begin{thm}\label{thm:generalfamilies}
  Let $\mu$ be a positive Borel measure and let $\eta$ be a distribution.
  Then there exists a family $\mu_s$ of positive measures with $\mu_0=\mu$ and derivative $d\mu_s/ds|_{s=0}=\eta$ if, and only if, $\eta$ satisfies the following condition:
  \begin{enumerate}
\item[{$\mathrm{\diffcondition}$}]   $\langle\eta,f\rangle=0$ for every nonnegative $f\in\C^\infty_c(P)$ that vanishes indentically on $\supp\mu$.
  \end{enumerate}
  If $\mu$ is a probability measure and $\eta$ additionally satisfies that $\langle\eta,1\rangle=0$, then $\mu_s$ can be realized as a family of probability measures.
 \end{thm}

\begin{rmk}\label{rmk:examplesandsupp}
 Condition \diffcondition\ implies that $\supp\eta\subseteq\supp\mu$.
 Apart from this, Condition \diffcondition\ is relevant only when $\supp\mu$ has parts that are very thin --- only one point thick. 
 
 For example, if $P=\R$, and if $\mu$ is the Dirac delta $\delta_0$, then Condition \diffcondition\ implies that $\eta$ must be of the form $A\delta_0+B\partial\delta_0$, $A,B\in\R$. Indeed, take a cutoff function $\rho\colon \R\to\R\in C^\infty_c(\R)$ (i.e., $\rho\geq0$, $\rho\equiv 1$ in a neighborhood of 0 and $\rho\equiv 0$ outside a slightly larger neighborhood). Then taking $f(x) = \rho(x)\sum_{i\geq2}c_i x^i$ (with $c_2$ large enough to ensure that $f\geq0$) we see that $\eta$ must be of the proposed form in order to comply with Condition \diffcondition.
 
 On the other hand, if we again had $P=\R$, but now $\mu=\chi_{[0,1]}$ the characteristic function on the unit interval, then as long as $\supp\eta\subseteq\supp\mu$, $\eta$ can be any distribution and still comply with Condition \diffcondition.
\end{rmk}

\begin{rmk}
 The family $\mu_s$ can always be realized as a family of smooth measures (except maybe at $s=0$). 
 Indeed, if $\mu_s$ is any family of measures that is differentiable at $s=0$, $\psi$ is a mollifier, and $F$ is a generating tuple, then the measure $\tilde\mu_s=\psi*_{s^2F}\mu_s$ has the same derivative at 0 and the same mass as $\mu_s$, and $\tilde\mu_s$ is a positive measure if $\tilde\mu_s$ is. By Lemma \ref{lem:measuredistribution}, the measure $\tilde\mu_s$ is a smooth density for all $s\neq 0$. 
\end{rmk}

 \begin{lem}\label{lem:variationpoint}
  Fix a point $p\in \supp\mu\subseteq P$. Let $\eta_p$ be a distribution supported on $p$ that satisfies Condition \diffcondition.
 Then there is a family of positive measures $\mu_{s}^p$ such that $\mu_0^p=\mu$ and
 \[\left.\frac{d\mu_{s}^p}{ds}\right|_{s=0}=\eta_p.\]
 Moreover, the dependence of $\mu_s^p$ on $p$ is measurable.
 
  If $\mu$ is a probability measure and additionally $\langle\eta_p,1\rangle=0$, then $\mu_s^p$ can be realized as a family of probability measures.
\end{lem}

For the proof of the lemma we will need a metric defined on the space of distributions involving up to $k^{\textrm{th}}$ derivatives, $k\geq 1$, and given by
\begin{equation}\label{eq:distk}
\dist_k(\theta_1,\theta_2)=\sum_{j=1}^\infty \frac{1}{2^j\|f_j\|_k}\left|\langle\theta_1,f_j\rangle-\langle\theta_2,f_j\rangle\right|
\end{equation}
for two distributions $\theta_1$ and $\theta_2$, and with $\{f_j\}_j\subset C^\infty_c(P)$  a sequence of functions that is dense with respect to the $C^k$ norm
\[\|f\|_k=\sum_{|I|\leq k}\sup_{q\in P}|\partial^If(q)|.\]

\begin{proof}[Proof of Lemma \ref{lem:variationpoint}]
Denote by $\posfuncs$ the space of nonnegative functions $f\in C^\infty_c(P)$ that vanish identically on $\supp \mu$.
Let $V\subseteq T_pP$ be the subspace that is null for the Hessians at $p$ of all the functions in $\posfuncs$:
\[V=\{\textrm{$v\in T_pP:\hessian_p f(v,v)=0$ for all $f\in\posfuncs$}\}.\]
Let $m=\dim V\leq n=\dim P$.
Take coordinates $(x_1,x_2,\dots,x_{n})$ around $p$ such that the vectors
\[\frac{\partial}{\partial x_1},\dots,\frac{\partial}{\partial x_{m}}\in T_pP\]
form a basis of $V$ and $\partial/\partial x_1,\dots,\partial/\partial x_{n}$ is an orthonormal basis of $T_pP$. 
Then by Lemma \ref{lem:decompositionmeasure} we know that $\eta$ must be a finite linear combination of distributions of the form
\begin{equation*}
\left(\frac{\partial}{\partial x_u}\right)^{e_0}
\left(\frac{\partial}{\partial x_1}\right)^{e_1}
\left(\frac{\partial}{\partial x_2}\right)^{e_2}
\cdots
\left(\frac{\partial}{\partial x_{m}}\right)^{e_{m}}
\delta_p
\end{equation*}
where $e_0\in\{0,1\}$, $u>m$, and the integers $e_1,\dots,e_{m}$ are nonnegative. For reasons analogous to those explained in Remark \ref{rmk:examplesandsupp}, Condition \diffcondition\ makes it impossible to have higher derivatives in the directions outside $V$ (i.e., in the direction of $x_u$ in this expression). 

Let us show that it is enough to prove the lemma for the case in which $e_0=0$. Indeed, if $\nu_s$ is a family of positive measures such that $\nu_0=\mu$ and 
\begin{equation}\label{eq:smallerdistribution}
 \left.\frac{d\nu_s}{ds}\right|_{s=0} = 
\left(\frac{\partial}{\partial x_1}\right)^{e_1}
\left(\frac{\partial}{\partial x_2}\right)^{e_2}
\cdots
\left(\frac{\partial}{\partial x_{m}}\right)^{e_{m}}
\delta_p,
\end{equation}
and if $\phi$ is the flow of the vector field $\partial/\partial x_u$, then 
\[\left.\frac{d}{ds}\phi_s^*\nu_s\right|_{s=0}=
\frac{\partial}{\partial x_u}
\left(\frac{\partial}{\partial x_1}\right)^{e_1}
\left(\frac{\partial}{\partial x_2}\right)^{e_2}
\cdots
\left(\frac{\partial}{\partial x_{m}}\right)^{e_{m}}
\delta_p.
\]
So we will assume that $e_0=0$ and we will focus on finding such a family $\nu_s$. In particular, we will assume that $\eta_p$ is of the form given in the right-hand-side of equation \eqref{eq:smallerdistribution}. In other words, we will assume that it only involves derivatives in the directions of $V$.

 \begin{lem}\label{lem:infiszero}
Let $\psi$ be a mollifier and $F$ a generating tuple of vector fields.
 For $\eta_p$ as in the right-hand-side of equation \eqref{eq:smallerdistribution} and for $k=\sum_{i=1}^{m}e_i$, we have, for all $t>0$,
 \[\inf_g\dist_k(g\cdot(\psi*_{tF}\mu),\eta_p)=0,\]
 where the infimum is taken over all measurable functions $g\colon P\to\R$, and $\dist_k$ is the distance defined in \eqref{eq:distk}.
\end{lem}
The reader will find the proof of Lemma \ref{lem:infiszero} below.

With $\psi$ and $F$ as in the lemma, let 
\[r_{i,m}=\sum_{j=i}^\infty \frac{1}{2^{jm}}\qquad\textrm{and}\qquad\mu_i=\psi*_{r_{i,2}^2 F}\mu\]
for $i\in\N$.
In particular $r_{i,1}\to0,r_{i,2}\to0$, and $\mu_i\to\mu$ as $i\to+\infty$. Let $k$ be as in Lemma \ref{lem:infiszero}, and denote by $\|\cdot\|_\infty$ de essential supremum norm.
 For each $j\in\N$, take a measurable function $g_j$ such that $\|g_j\|_\infty\leq 1$ and 
\[\dist_k(2^{j}g_j\mu_j,\eta_p)<\frac1{2^j}+\inf_{\|g\|_\infty\leq 1}\dist_k(2^{j}g\mu_j,\eta_p),\]
where the infimum is taken over all measurable functions $g:P\to\R$ with essential supremum $\leq 1$.
With this definition, Lemma \ref{lem:infiszero} implies that if we let $j\to+\infty$, we get $2^{j}g_j\mu_j\to\eta_p$.

We let, for  $r_{i+1,2}\leq |s|<r_{i,2}$,
\[\nu_s=(1-2^{1-i})\mu_{i}+\sum_{j=i}^\infty\frac{1}{2^j}(1+g_j\sgn s)\mu_j.\]
By construction, $\nu_s$ is a family of positive measures such that $\nu_s\to\mu$ as $s\to0$ and its derivative at $s=0$ is $\eta_p$. To see why, first note that, as $s\to0$, we have $i\to+\infty$,
\[(1-2^{1-i})\mu_i+\sum_{j=i}^\infty\frac{\mu_j}{2^j}\to\mu,\] 
and the derivative of that term at $s=0$ vanishes by Lemma \ref{lem:convolutionderivative}. 
The other term vanishes as $s\to0$, and its derivative is the limit, as $s\to0$, of
\[\frac1s\sum_{j=i}^\infty\frac{\sgn s}{2^j}g_j\mu_j\approx 
\frac{1}{r_{i,2}}\sum_{j=i}^\infty\frac1{2^j}g_j\mu_j\underset{\textrm{L'H\^opital}}{\to}\frac{2^{-i}g_i\mu_i}{2^{-2i}}\to\eta_p,
\]
where we applied L'H\^opital's rule because both the sum and $r_{i,2}$ tend to 0 as $s\to0$ (or $i\to+\infty$).

To ensure the measurability of the $p$-dependence of this construction, we further specify the construction as follows. For each $j\in\Z_-$, we take a covering of $P$ by measurable sets $A_j$ of diameter at most $-1/j$. For all $p\in A_i$, we take the same function $g_j$. This ensures that these choices are made in a `measurable' way. The rest of the construction does not depend on arbitrary choices, so the dependence becomes measurable.

The last statement of the lemma follows from the fact that if $\eta_p$ satisfies $\langle\eta_p,1\rangle=0$, then either $g_j$ can be chosen so that $g_j\mu_j$ satisfies this too, or else $e_1=e_2=\dots=e_{n}=0$, and in both cases the coordinates can be picked so that the mass is preserved by the flow $\phi_s$ for small-enough $|s|$.
\end{proof}

\begin{proof}[Proof of Lemma \ref{lem:infiszero}]
This is a local problem and by pushing forward with a chart, we may assume that $P$ is some Euclidean space $\R^n$.
Let $\bar\mu=\psi*_{tF}\mu$. Note that $\supp\bar\mu$ has nonempty interior, and in fact contains a neighborhood of $p$.

 For each $j=1,2,\dots$, let $\{x^j_i\}_{i=1}^\infty\subseteq \supp\bar\mu\subset \R^{n}$ be a sequence of points contained within distance $1/j$ of $p$. We also assume that their Zariski closure is all of $\R^n$ (i.e., that no nonzero polynomial vanishes on all of them simultaneously). For a large-enough finite subset $I_j$ of $\N$, there is always a solution to the problem of finding real numbers $c_{ij}$ such that
 \begin{equation}\label{eq:weightsdef}
 \langle \eta_p,f\rangle=\lim_{h\to0}\frac{1}{h^k}\sum_{i\in I_j}c_{ij}f(h x^j_i)
 \end{equation}
 for all $f\in C^\infty_c(P)$.
 To see this, note that expanding the right-hand-side as Taylor series in $h$ and comparing coefficients, one obtains a linear system in the variables $c_{ij}$, and that this system has solutions if sufficiently many points $x^j_i$ are available and if they are in a sufficiently-general position (which we may assume to be true since their Zariski closure can be made as large as necessary). Equivalently, we have a measure that approximates $\eta_p$: 
 \[h^{-k}\sum_{i\in I_j}c_{ij}\delta_{x^j_i}\to\eta_p\quad \textrm{as}\quad h\to0.\]
 These measures also tend to $\eta_p$ as $j\to\infty$.

 We now approximate those measures with measurable functions.
 For each $j=1,2,\dots$, let $0<\varepsilon_j<j^{-2}$ be small enough that the balls $B_{\varepsilon_j}(x^j_i)$ are disjoint. For $q\in B_{\varepsilon_j}(x^j_i)\cap\supp\bar\mu$ for some $i\in I_j$, let
 \[g_j(q) = \frac{c_{ij}}{\bar\mu(B_{\varepsilon_j}(x^j_i))},\]
 and let $g_j(q)=0$ for all other $q\in P$. Then $g_j\bar\mu\to\eta_p$, and this proves the lemma.
\end{proof}

\begin{proof}[Proof of Theorem \ref{thm:generalfamilies}]
 Assume first that the family $\mu_s$ exists. 
 To prove that Condition \diffcondition\ must hold, let $f\in \posfuncs$ (as defined in the proof to Lemma \ref{lem:variationpoint}), and consider the function
 \[g(s)=\int f\,d\mu_s.\]
 Since $f$ is nonnegative and $\mu_s$ is a positive measure for all $s$, $g$ must be nonnegative as well. Since $g(0)=0$, it must also be true that $g'(0)=0$, and this is equivalent to Condition \diffcondition.
 
 Now assume that we have a measure $\mu$ and a distribution $\eta$ such that Condition \diffcondition\ holds, and let us construct a family $\mu_s$ as in the statement of the theorem. Write $\eta=\sum_I\partial^I\nu_I$ as in Lemma \ref{lem:decompositionmeasure} where the $\nu_I$ are signed measures for each multi-index $I$, and the sum is locally finite. 
 For all $I$ with $|I|>0$ we have
 \begin{multline*}
 \langle\partial^I\nu_I,f\rangle=(-1)^{|I|}\int\partial^Ifd\nu_I
 =(-1)^{|I|}\int\int\partial^If\,d\delta_p\,d\nu_I(p) \\
 =\int\langle\partial^I\delta_p,f\rangle\,d\nu_I(p).
 \end{multline*}
 We take 
 \[\eta_{p,I}=\partial^I\delta_p.\]
 For $\nu_I$-almost all $p$, the distributions $\eta_{p,I}$ also satisfy Condition \diffcondition.
 From Lemma \ref{lem:variationpoint}, we get families $\mu_s^{p,I}$ of measures whose derivatives at 0 are precisely the distributions $\eta_{p,I}$. Thus by linearity of the derivative,
 \[\mu_s=\mu+s\nu_\varnothing+\sum_{|I|>0}\int\mu_s^{p,I}d\nu_I(p)\]
 is a family as in the statement of the theorem in an neighborhood of $s=0$, and can be easily modified to satisfy it for all $s\in\R$.
 
 If $\mu$ is a probability, since each $\mu^{p,I}_s$ preserves the probability for $|I|>0$ and since $\langle \eta,1\rangle=0$ implies that $\nu_\varnothing(P)=0$, we conclude that $\mu_s$ also preserves the probability.
\end{proof}

\subsubsection{One-sided derivatives}
\label{sec:onesided}

Since one-sided differentiability is a less restrictive condition than two-sided differentiability, if we allow signed measures we will again get that all distributions arise as derivatives of such families. Thus Proposition \ref{prop:characterizationsigned} holds word-for-word for one-sided derivatives too.

In the case of positive measures, a small modification to Theorem \ref{thm:generalfamilies} is necessary:

\begin{thm}\label{thm:onesidedfamilies}
Let $\mu$ be a positive Borel measure and let $\eta$ be a distribution. Then there exists a family $\mu_s$, $s\geq 0$, of positive measures $\mu_0=\mu$ and one-sided derivative $\eta=d\mu_s/ds|_{s=0+}$ if, and only if, $\eta$ satisfies the following condition:
\begin{enumerate}
\item[$(\mathrm{Pos}^+)$] $\langle\eta,f\rangle\geq 0$ for every nonnegative $f\in C^\infty_c(P)$ that vanishes identically on $\supp(\mu)$.
\end{enumerate}
If $\mu$ is a probability measure and $\eta$ additionally satisfies that $\langle\eta,1\rangle=0$, then $\mu_s$ can be realized as a family of probability measures.
\end{thm}
\begin{proof}[Sketch of proof]
Any distribution $\eta$ that satisfies Condition $(\mathrm{Pos}^+)$ can be written as 
\[\eta=\eta^0+\kappa,\]
where $\eta^0$ satisfies \diffcondition\ and $\kappa$ is a positive Borel measure on $P$ (with no restrictions on $\supp\kappa$). A family $(\mu_s)_s$ can thus be produced using the techniques used to prove Theorem \ref{thm:generalfamilies}.
\end{proof}

\subsection{Flows}
\label{sec:flows}

In this section we aim to give a rough scheme of how one can find an object closely related to the Colombeau algebra that gives a sense of ``direction of the movement'' for many distributions, in direct connection with mass transport theory.

For simplicity, let $P=\R^n$, $n\geq 1$. When $\mu_t$ is a family of densities that defines an absolutely continuous curve in Wasserstein space, it has been shown (see for example \cite[Chapter 8]{ambrosiogiglisavare}) that the derivatives of $\mu$ can be interpreted as the divergence of a vector field, that is, there are vector fields $v_t$ on $\R^n$ satisfying the continuity equation,
\begin{equation}\label{eq:continuity}
\frac{d\mu_t}{dt}+\divergence(\mu_tv_t)=0, \quad t\in\R.
\end{equation}
The interpretation is that the mass of $\mu_t$ is being transported by the flow of the vector field $v_t$.
This gives a way to assign a vector field $v_t$ to the  distribution $d\mu_t/dt$, and this vector field gives a notion of ``direction of the movement.'' The vector field $v_t$ is not unique; it is ambiguous by a vector field $u_t$ such that $\divergence(\mu_tu_t)=0$ for all $t\in\R$. Since for norm-induced topologies on the space of vector fields the set of possibilities is closed, one can choose a norm and choose the $v_t$ to be the minimizer for each $t$. One can show that the minimizer is in fact a gradient vector field, $v_t=\nabla \phi_t$ for some functions $\phi_t\colon\R^n\to\R$.

On the other hand, the results of Section \ref{sec:general_deformations} indicate that in the case of more general curves that are not absolutely continuous with respect to the Wasserstein metric (but are differentiable in the sense considered in Section \ref{sec:general_deformations}), much more general distributions can arise as the derivative. Thus, we can see an arbitrary distribution $\eta$ as the derivative of a family of measures at some point, say, $d\mu_t/dt|_{t=0}=\eta$, and we can use the continuity equation \eqref{eq:continuity} to try to assign an object that will give an idea of direction of the movement determined by $\eta$.

This can be done using Colombeau algebras. These algebras were developed \cite{colombeau} to provide a context in which distributions can be multiplied. All distributions are contained in these algebras. Roughly speaking, the solution to the multiplication problem is to record, instead of the distribution itself, all possible smoothings of the distribution. An equivalence relation is then proposed on a certain set of families of smooth functions, and its equivalence classes are the elements of the algebra. 

To define the relevant Colombeau algebra, we follow \cite[Section 8.5]{colombeaubook}.
Let $\mathcal E(P)$ be the set of families $(f_\varepsilon)_{0<\varepsilon<1}$ of functions $f_\varepsilon\in C^\infty(P)$ indexed by $0<\varepsilon<1$, such that for each compact set $K\subset P$ and every multi-index $I$ there are $N\in\N$, $\eta>0$, and $c>0$ such that 
\[\sup_{x\in K}\left|\partial^If_\varepsilon(x)\right|\leq \frac{c}{\varepsilon^N}\quad \textrm{if $0<\varepsilon<\eta$}.\]
We define the ideal $\mathcal N(P)$ of $\mathcal E(P)$ to be the set of families $(f_\varepsilon)_{0<\varepsilon<1}$ such that for all compact sets $K$, for all multi-indices $I$, and for all $q\in\N$ there exist $c>0$ and $\eta>0$ such that
\[\sup_{x\in K}\left|\partial^If_\varepsilon(x)\right|\leq c\varepsilon^q\quad\textrm{if $0<\varepsilon<\eta$}.\]
This means that the objects in $\mathcal N(P)$ have a fast decay (faster than any power of $\varepsilon$) when $\varepsilon\to0$. The \emph{Colombeau algebra} is the quotient
\[\mathcal G(P)=\mathcal E(P)/\mathcal N(P).\]
All distributions $\eta$ are contained in $\mathcal G(P)$ because the families $(\psi*_{\varepsilon F}\eta)_\varepsilon$ of smoothings are contained there.
We denote by $[\eta]$ the set of elements of $\mathcal G(P)$ that would be associated to the distribution $\eta$; that is,  $(f_\varepsilon)_\varepsilon\in\mathcal G(P)$ belongs to $[\eta]$ if for all $\phi\in C^\infty_c(P)$
\[\lim_{\varepsilon \to 0+}\int \phi(x)\, f_\varepsilon(x)\,dx-\langle \eta,\phi\rangle=0.\]

Let $\mu$ be a Borel measure on $P$ and let $\eta$ be a distribution satisfying \diffcondition, so that the conclusions of Theorem \ref{thm:generalfamilies} hold. 
Let $\mathcal M(\mu)$ be the set families of smooth vector fields $(v_\varepsilon)_{0<\varepsilon<1}$, such that for all $(\mu^\varepsilon)_\varepsilon\in[\mu]$ satisfying $\int \mu^\varepsilon(x)\,dx=1$ and $\mu^\varepsilon(x)\geq 0$ for $0<\varepsilon<1$, we have
\[(\divergence(\mu^\varepsilon v_\varepsilon))_\varepsilon\in\mathcal G(P).\]
The space $\mathcal{M}(\mu)$ is clearly a vector space, and it always contains a solution to the system of equations
\[\eta^\varepsilon+\divergence(\mu^\varepsilon v_\varepsilon)=0\]
where $(\eta^\varepsilon)_\varepsilon\in[\eta]$ is such that for all $0<\varepsilon<1$ we have $\supp \eta^\varepsilon\subseteq\supp\mu^\varepsilon$ (such $(\eta^\varepsilon)_\varepsilon$ always exists because $\eta$ satisfies \diffcondition) because this is just the classical case of equation \eqref{eq:continuity}. 

The (non-unique) family $(v_\varepsilon)_\varepsilon\in \mathcal M(\mu)$ is the object we have been pursuing, as it gives precise meaning to the notion of ``direction of movement'' associated to $\eta$ with respect to $\mu$.
It would be interesting to know which parts of the theory of mass transport still hold in this context. In particular, it is not clear whether the hypothetical object ``$\lim_{\varepsilon\to0}v_\varepsilon$'' itself corresponds to a distribution on $TP$ in some cases.

%% file: setting.tex
\paragraph{Phase space.}

Let $M$ be a compact, oriented $C^\infty$ manifold of dimension $d\geq 1$, without boundary $\partial M=\varnothing$. 
Denote by $TM$ its tangent bundle and, for $n\geq 1$, denote by $T^nM$ the direct sum bundle
\[T^nM=\underbrace{TM\oplus\cdots\oplus TM}_n\]
of $n$ copies of $TM$.
The dimension of $T^nM$ is $d(n+1)$. 
An element in $T^nM$ can be denoted $(x,v_1,v_2\dots,v_n)$, where $x$ is a point in $M$ and $v_1,v_2,\dots, v_n\in T_xM$ are vectors tangent to $x$.
When taking local coordinates, we will write
\[x=(x_1,x_2,\dots,x_d)\quad\textrm{and} \quad v_i=(v_{i1},v_{i2},\dots,v_{id}).\]
Sometimes for brevity we will write $(x,v)$ instead of $(x,v_1,v_2,\dots,v_n)$.

The projection $\pi:T^nM\to M$ is given by $\pi(x,v_1,\dots,v_n)=x$. We denote by $\Omega^n(M)$ the space of smooth differential $n$-forms on $M$. We will often consider these forms as smooth functions on $T^nM$.

Throughout, when referring to functions on these objects, we will use the term \emph{smooth} to mean $C^\infty$. We will denote by $C^\infty(X,Y)$ the space of all smooth functions $X\to Y$.  If $Y$ is the real line $\R$, we will sometimes omit it in our notation. We will denote by $C_c^\infty(X)$ the set of all real-valued, compactly-supported, smooth functions on the set $X$.

\paragraph{Riemannian structure.} 
We fix, once and for all, a Riemannian metric $g\in C^\infty(T^2M)$ on $M$ and its corresponding Levi-Civita connection $\nabla$. We denote the operation of covariant differentiation in the direction of a vector field $F$ by $\nabla_F$. 

We will denote $|v|=\sqrt{g(v,v)}$ for $v\in T_xM$, and we extend this norm to $T^n_xM$ by letting
\[|(v_1,v_2,\dots,v_n)|=\sqrt{|v_1|^2+|v_2|^2+\cdots+|v_n|^2}.\]

\paragraph{Forms.}
We will denote by $\Omega^k(M)$ the space of smooth differential $k$-forms on $M$. On this space we define a norm $\|\cdot\|$ by letting, for $\omega\in \Omega^k(M)$,
\begin{align*}
\|\omega\|&=\sup\{\omega_x(v_1,\dots,v_k):(x,v_1,\dots,v_k)\in T^nM,|v_i|\leq 1\} \\
&=\left|\det(g(v_i,v_j)_{i,j=1}^n\right|
\end{align*}

%% file: measures.tex

For $x\in M$, define the volume function $\vol_n:T^n_xM\to \R$ by
\[\vol_n(v_1,\dots,v_n)=\sup_{\omega\in \Omega^n(M),\|\omega\|\leq1}\omega_x(v_1,\dots,v_n).\]

We let $\Pn$ be the space of \emph{subvolume functions}, that is, the space of real-valued continuous functions $f\in C^0(T^nM)$ such that
\[\sup_{(x,v_1,\dots,v_n)\in T^nM}\frac{|f(x,v_1,\dots,v_n)|}{1+\vol_n(v_1,\dots,v_n)}<+\infty.\]
Note that all differential $n$-forms on $M$ belong to $\Pn$ when regarded as functions on $T^nM$. We endow $\Pn$ with the supremum norm and its induced topology.

\paragraph{Mild measures.}
We define the \emph{mass} $\mass(\mu)$ of $\mu\in \Mn$ to be
\[ \mass(\mu)
 =\int_{T^nM}\vol_n(v_1,v_2,\dots,v_n)\,d\mu(x,v_1,\dots,v_n).
\]
This is always a nonnegative number.
A positive Borel measure $\mu$ on $T^nM$ is \emph{mild} if $\mass(\mu)<+\infty$. Denote by $\Mn$ the space of mild measures.

The space $\Mn$ is natually embedded in the dual space $\Pn^*$ and we endow it with the topology induced by the weak* topology on $\Pn^*$. Although the topology on $\Pn^*$ is not metrizable, the topology on $\Mn$ is. 
We can give a metric in $\Mn$ by picking a sequence of functions $\{f_i\}_{i\in\N}\subset\Cinfc$ that are dense in $\Pn$, and then letting
\begin{equation}\label{eq:mildmetric}
\dist_\Mn(\mu_1,\mu_2)=|\mass(\mu_1)-\mass(\mu_2)|+\sum_{k=1}^\infty\frac1{2^k\sup|f_k|}\left|\int|f_k|d\mu_1-\int|f_k|d\mu_2\right|.
\end{equation}

\paragraph{Holonomic measures.}
A mild measure $\mu\in \Mn$ is \emph{holonomic} if it is a probability (that is, a positive measure such that $\mu(T^nM)=1$), and if for every differential $(n-1)$-form $\omega\in\Omega^{n-1}(M)$,
\begin{equation}\label{eq:holonomicitycondition}
\int_{T^nM}d\omega\,d\mu=0.
\end{equation}
The space $\holonomic$ of holonomic measures is convex. 

The motivation for this definition is given by Proposition \ref{prop:holonomic} below. 
\begin{rmk}
A mild measure $\mu$ induces a current $T\colon \Omega^n(M)\to\R$ given by
\[\langle T,\omega\rangle=\int \omega\,d\mu.\]
Motivated by Stokes's theorem, the boundary $\partial T$ of this current is defined by duality as
\[\langle \partial T,\omega\rangle=\langle T,d\omega\rangle.\]
The definition of holonomic measures equivalent to requiring the boundary $\partial T$ of the induced current to be empty.
\end{rmk}

\paragraph{Cellular complexes.}
An \emph{$n$-dimensional cell} (or \emph{$n$-cell}) $\gamma$ is a smooth map
\[\gamma:D\subseteq\R^n\to M,\]
where $D$ is a subset of $\R^n$ homeomorphic to a closed ball, \emph{together} with a choice of coordinates $t=(t_1,t_2,\dots,t_n)$ on $D$. A \emph{chain} of $n$-cells is a formal linear combination of the form 
\[a_1\gamma_1+a_2\gamma_2+\cdots+a_k\gamma_k\]
for real numbers $a_1,a_2,\dots,a_k$ and $n$-cells $\gamma_1,\gamma_2,\dots,\gamma_k$. We will say that a chain is \emph{positive} if $a_i>0$. 

Let $\gamma:D\subseteq\R^n\to M$ be an $n$-cell. Denote by $d\gamma$ the differential map associating, to each element in $D$, an element in $T^nM$. Explicitly, if we have coordinates $t=(t_1,t_2,\dots,t_n)$ on $D$, then
\[d\gamma(t)=\left(\gamma(t),\frac{\partial\gamma}{\partial t_1}(t),\frac{\partial\gamma}{\partial t_2}(t),\cdots,\frac{\partial\gamma}{\partial t_n}(t)\right).\]
This map \emph{does} depend on our choice of coordinates $t$.

To an $n$-cell $\gamma$, we associate a measure $\mu_\gamma$ on $T^nM$ defined by 
\[\int_{T^nM}f\,d\mu_\gamma=\int_Df(d\gamma(t))\,dt,\]
where $dt=dt_1\wedge\cdots\wedge dt_n$. In other words, $\mu_\gamma$ is the pushforward of Lebesgue measure on $D$, $\mu_\gamma=d\gamma_*dt$.

Similarly, to a chain of $n$-cells $\alpha=\sum_{i=1}^k a_i\gamma_i$, we associate the measure $\mu_\alpha$ given by 
\[\mu_\alpha=\sum_{i=1}^ka_i\mu_{\gamma_i}.\]
The measure $\mu_\alpha$ is an element of $\Mn$. We will say that the chain $\alpha$ is a \emph{cycle} if for all forms $\omega\in \Omega^{n-1}(M)$,
\[\int_{T^nM}d\omega\,d\mu_\alpha=0.\]
That is to say, that the chain $\alpha$ is a cycle if $\mu_\alpha$ is holonomic.

\begin{prop}\label{prop:holonomic}
 Assume that $1\leq n\leq d$. Let $\mu\in\Mn$ be a probability measure on $T^nM$. Then the following conditions are equivalent:
 \begin{enumerate}
  \item The measure $\mu$ is holonomic.
  \item There exists a sequence $\{\alpha_k\}_{k\in\N}$ of cycles such that the induced measures $\mu_{\alpha_k}$ are all probabilites on $T^nM$, and $\mu_{\alpha_k}\to\mu$ as $k\to\infty$ in the topology induced by the distance \eqref{eq:mildmetric}.\qed
 \end{enumerate}
\end{prop}
This was proved in {\cite{myclosedmeasuresareholonomic}}.
Thus the space of holonomic measures is precisely the closure of the space of measures $\mu_\alpha$ induced by cycles $\alpha$.

%% file: distributions.tex
%
%

A \emph{partition of unity} in $T^nM$ is a set of nonnegative functions $\{\psi_i\}_i\subset \Cinfc$ such that for all $x\in T^nM$ 
\[\sum_i\psi_i(x)=1.\]
Recall that distributions on manifolds were discussed in Section \ref{sec:subdistributions}.
Given a distribution $\eta\in \distributions$, we want to make sense of its value at a form $\omega\in \Omega^n(M)$. We let
\[\langle\eta,\omega\rangle=\sum_i\langle\eta,\psi_i\omega\rangle,\]
We denote by $\Dn\subset\distributions$ the set of \emph{mild distributions}, namely, the set of those distributions for which the series in the right-hand-side converges absolutely for all $\omega\in \Omega^n(M)$. This is independent of our choice of partition of unity $\{\psi_i\}_i$. Also, the spaces of mild measures $\Mn$ and of holonomic measures $\holonomic$ are subsets of $\Dn$.

A family of measures $\mu_t \in \Mn$ is \emph{differentiable} at 0 if there is a distribution $\eta\in\Dn$ such that for all $f\in\Cinfc$ 
\[\left.\frac{d}{dt}\right|_{t=0}\int f\,d\mu_t=\langle\eta,f\rangle.\]

%% file: tangent.tex
\begin{thm}\label{thm:tangent}
 Let $\mu$ be a holonomic measure in $T^nM$ and let $\eta\in \Dn$ be a mild distribution on $T^nM$. Then there exists a family of holonomic measures $(\mu_t)_{t\in\R}\subset\Mn$ such that $\mu_0=\mu$ and
 \begin{equation}\label{eq:tangentvector}
\ddtzero\int f\,d\mu_t=\langle\eta,f\rangle
 \end{equation}
 for all $f\in \Cinfc$ if, and only if, the following conditions are satisfied:
\begin{enumerate}
\item[\textup{\diffcondition}]\label{it:conditionD} For all nonnegative $f\in \Cinfc$ that vanish on $\supp\mu$, $\langle\eta,f\rangle=0$.
\item[\textup{\holcondition}]\label{it:conditionH} For all differential forms $\omega \in \Omega^{n-1}(M)$, $\langle\eta,d\omega\rangle=0$.
\item[\textup{\probcondition}] $\langle\eta,1\rangle=0$.
\end{enumerate}
\end{thm}
\begin{rmk}
 In other words, the tangent space to the space of holonomic measures at the point $\mu$ is characterized by Conditions \diffcondition, \holcondition, and \probcondition. In fact, we have the following easy
 \begin{lem}\label{lem:Poseasy}
 If $\langle\eta,f\rangle$ consists of an integral of $f\in \Cinfc$ and of an integral of first derivatives of $f$ (i.e., if $\eta$ extends to a functional on $C^1_c(T^nM)$), and if $\supp\eta\subseteq\supp\mu$, then $\eta$ satisfies condition \textup{\diffcondition}.\qed
 \end{lem}
\end{rmk}

\begin{proof}[Proof of Theorem \ref{thm:tangent}]
 By Theorem \ref{thm:generalfamilies}, Condition \diffcondition\ is necessary. If $(\mu_t)_t$ exists, then we have
 \[0=
 \left.\frac{d}{dt}\right|_{t=0}\int d\omega\,d\mu_t=\langle\eta,d\omega\rangle\]
 for all $\omega\in\Omega^{n-1}(M)$. Hence, Condition \holcondition\ is also necessary. Condition \probcondition\ is necessary because we want $\mu_t(T^nM)=1$ for all $t$.
 
 To prove that Conditions \diffcondition, \holcondition, and \probcondition\ are sufficient, assume that they are satisfied. Then by Theorem \ref{thm:generalfamilies} we have a family of probability measures $\theta_t$ for $t$ in some interval that contains 0, with $\theta_0=\mu$ and with \eqref{eq:tangentvector}.  Moreover, the proofs of Theorem \ref{thm:generalfamilies} and Lemma \ref{lem:variationpoint} show that $\theta_t$ can be assumed to be in $\Mn$ for all $t$. Now we need to modify $\theta_t$ so that it is also a family of holonomic measures.
 
 There exists a family of positive measures $\nu_t$ such that for all $\omega\in \Omega^{n-1}(M)$ and all $t$
 \[\int d\omega\,d\theta_t + \int d\omega\,d\nu_t=0.\]
 The measure $\nu_t$ can for example be obtained from $\theta_t$ as follows. For each $x\in M$, let $r_x:T^n_xM\to T^n_xM$ be some reflection such that the multivector $r_x(v)$ has the opposite orientation as the multivector $v\in T^nM$. These reflections can be chosen in a piecewise-continuous (and hence measurable) way with respect to the variable $x$. Then one can take the family of measures determined by $\nu_t|_{T_x^nM}=r_x^*(\theta_t|_{T_x^nM})$.
 
 We may assume that $\nu_t(T^nM\cap\{0\})=0$ because the part of $\nu_t$ on the zero section does not contribute to the integrals 
 \[\int d\omega\,d\nu_t,\quad \omega\in \Omega^{n-1}(M).\]
 
 For $a>0$, let $\lambda_a:T^nM\to T^nM$ be the map given by
 \[\lambda_a(x,v_1,v_2,\dots,v_n)=(x,av_1,av_2,av_3,\dots,av_n).\]
 The measure $\nu_t^a=\lambda_a^*\nu_t/a^n$ satisfies
 \[\int d\omega\,d\nu_t^a=\frac1{a^n}\int d\omega(x,av_1,\dots,av_n)\,d\nu_t=\int d\omega\,d\nu_t\]
 for all $\omega\in\Omega^{n-1}(M)$. As $a\to\infty$, the mass $\int d\nu_t^a$ of $\nu^a_t$ tends to 0. It is hence possible to find a function $b:\R-\{0\}\to\R_+$ such that $\nu_t^{b(t)}$ is a family of measures with 
 \[\left.\frac{d\nu_t^{b(t)}}{dt}\right|_{t=0}=0\qquad\textrm{and}\qquad \lim_{t\to0}\frac{1}{t^2} \int d\nu_t^{b(t)}=0.\]
 
 We let 
 \[\mu_t=\frac{\theta_t+\nu_t^{b(t)}}{1+\int d\nu_t^{b(t)}}\] 
 for $t\neq 0$ and $\mu_0=\mu$. This is a family of measures as in the statement of the theorem.
\end{proof}

%% file: onesidedhol.tex
We state the analogue of Theorem \ref{thm:tangent} for one-sided derivatives, which follows from \ref{thm:onesidedfamilies} in a similar way as Theorem \ref{thm:tangent} follows from Theorem \ref{thm:generalfamilies}.

\begin{thm}\label{thm:onesidedhol}
 Let $\mu$ be a holonomic measure in $T^nM$ and let $\eta\in \Dn$ be a distribution on $T^nM$. Then there exists a family of holonomic measures $\mu_t\in\Mn$, $t\geq0$, such that $\mu_0=\mu$ and
 \begin{equation}
\left.\frac{d}{dt}\right|_{t=0+}\int f\,d\mu_t=\langle\eta,f\rangle
 \end{equation}
 for all $f\in \Cinfc$ if, and only if, the following conditions are satisfied:
\begin{enumerate}
\item[\textup{$(\textrm{Pos}^+)$}]\label{it:conditionDplus} For all nonnegative $f\in \Cinfc$ that vanish on $\supp\mu$, $\langle\eta,f\rangle\geq0$.
\item[\textup{\holcondition}]\label{it:conditionHplus} For all differential forms $\omega \in \Omega^{n-1}(M)$, $\langle\eta,d\omega\rangle=0$.
\item[\textup{\probcondition}] $\langle\eta,1\rangle=0$.
\end{enumerate}
\end{thm}

Analogously to Lemma \ref{lem:Poseasy}, we have
\begin{lem}\label{lem:Pospluseasy}
 If the distribution $\eta$ extends to a functional on $C^1_c(\supp\mu) \cap C^0_c(T^nM)$ (i.e., if $\eta$ only involves an integral of the test function on $T^nM$ and an integral of its first derivatives on $\supp\mu$), then $\eta$ satisfies condition \textup{$(\textrm{Pos}^+)$}.
\end{lem}

%% file: examples.tex
Results in this section are valid for measures that are critical with respect to the action of a general smooth Lagrangian $L\in C^\infty(T^nM)$. Unless explicitly stated, we do not require, for example, that $L$ be convex. 

A \emph{variation} of a holonomic measure $\mu\in\holonomic$ is a family $\mu_t$ of holonomic measures that is defined for $t\geq0$, is differentiable at 0, and satisfies $\mu_0=\mu$.

We denote by $A_L$ the action of the Lagrangian $L$,
\[A_L(\mu)=\int_{T^nM} L\,d\mu.\]
We say that $\mu\in\holonomic$ is \emph{critical} for $A_L$ if for every variation $\mu_t$ with $\mu_0=\mu$ the one-sided derivative satisfies
\begin{equation}\label{eq:criticality}
 \left.\frac{d}{dt}\right|_{t=0^+}A_L(\mu_t)\geq0.
\end{equation}
By Theorem \ref{thm:onesidedhol}, $\mu$ is critical if, and only if, for all distributions $\eta\in\Dn$ that satisfies Conditions $\mathrm{(Pos^+)}$, \holcondition, and \probcondition, we have 
\begin{enumerate}
\item[\critcondition]$\langle\eta,L\rangle\geq0$.
\end{enumerate}
\begin{rmk}
Note that if $\eta$ satisfies not only $\mathrm{(Pos^+)}$, \holcondition, and \probcondition, but also \diffcondition (i.e., if it appears as a two-sided derivative), then \critcondition\ is equivalent to $\langle\eta,L\rangle=0$.
\end{rmk}

\paragraph{Homology.}
A holonomic measure $\mu\in\holonomic$ is assigned its \emph{homology class} $\rho(\mu)\in H_n(M;\R)$ by requiring
\[\langle\rho(\mu),\omega\rangle=\int\omega\,d\mu\]
for all closed forms $\omega\in\Omega^n(M)$, $d\omega=0$.
If for each $t$ the measure $\mu_t$ has the same associated homology class as $\mu_0$, $\rho(\mu_t)=\rho(\mu_0)$, then we say that the variation $\mu_t$ is \emph{homology preserving}. 
Clearly, for this to happen the following condition is necessary on the one-sided derivative $\eta=d\mu_t/dt|_{t=0^+}$:
\begin{enumerate}
\item[\homcondition] $\langle\eta,\omega\rangle=0$ for all $\omega\in \Omega^n(M)$ with $d\omega=0$.
\end{enumerate}

\begin{conj}
 Conditions $\mathrm{(Pos^+)}$, \holcondition, \probcondition, and \homcondition\ are sufficient for the existence of a homology preserving variation $\mu_t$.\qed
\end{conj}

We will say that $\mu\in\holonomic$ is \emph{critical for $A_L$ within its homology class} if equation \eqref{eq:criticality} holds for every homology preserving variation $\mu_t$ of $\mu$. In particular, if $\mu$ is critical for $A_L$, then it is also critical within its homology class.

%% file: horizontal.tex
Let $X\colon M\to TM$ be a smooth vector field on $M$. For $f\in\Cinfc$, denote by $Xf$ the Lie derivative in the (horizontal) direction $X$. This is given by $Xf=d_xf(X)$, and is independent of the Riemannian metric on $M$. For a differential form $\omega\in \Omega^n(M)$, the action of $X$ on $\omega$ is also defined, and it is equal to the Lie derivative $\mathcal L_X\omega=i_Xd\omega+di_X\omega$. Here, $i_X$ denotes the contraction.

Let $\mu$ be a holonomic measure on $T^nM$. The distribution $\eta$ given by
\begin{equation}\label{eq:horizontaldistribution}
 \langle\eta,f\rangle=\int_{T^nM} Xf\,d\mu
\end{equation}
for $f\in\Cinfc$ clearly satisfies Conditions \diffcondition\ and \probcondition. It also satisfies Condition \holcondition\ because for all $\omega\in\Omega^{n-1}(M)$,
\[\langle\eta,d\omega\rangle=\int \mathcal L_Xd\omega\,d\mu=\int i_Xd^2\omega+di_Xd\omega\,d\mu=0.\]
Therefore, $\eta$ is in the tangent space to $\mu$.

It also satisfies Condition \homcondition\ because, if $\omega$ is a closed $n$-form,
\[\frac{d}{ds}\int \omega\,d\mu_s=\int \mathcal L_X\omega\,d\mu_s=\int i_Xd\omega+di_X\omega\,d\mu_s=0.\]
 The last equality is true since $d\omega=0$ because $\omega$ is closed, and $\int di_X\omega\,d\mu_s=0$ because $\mu_s$ is holonomic.
 
In fact, it is easy to explicitly construct a family $\mu_t$ with derivative $\eta$ and $\mu_0=\mu$. To do this, take the flow $\phi_t:\R\times M\to M$ of $X$ on $M$, determined by
\[\phi_0(x)=x,\quad \frac{d}{dt}\phi_t(x)=X(x),\quad\textrm{for  $x\in M, t\in \R$.}\]
Extend this to an isotopy $r:\R\times T^nM\to T^nM$ by
\[r_t(x,v_1,\dots,v_n)=(\phi_t(x),d\phi_t(v_1),\dots,d\phi_t(v_n)),\]
where $d\phi_t:T_xM\to T_{\phi_t(x)}M$ denotes the derivative of $\phi_t$ at $x$. Then we can simply let $\mu_t=r_t^*\mu$. From this construction and Proposition \ref{prop:holonomic}, it is clear that $\mu_t$ is homology preserving. We thus have

\begin{prop}
 If $\mu$ is critical for $A_L$ within its homology class, then Condition \critcondition\ must hold for all distributions $\eta$ of the form given in equation \eqref{eq:horizontaldistribution}.\qed
\end{prop}

\paragraph{Euler-Lagrange equations.}
Assume that the holonomic measure $\mu$ is induced by a piecewise smooth cycle $\alpha$, that is,
\[\mu=\mu_\alpha.\]
We will now recover the traditional Euler-Lagrange equations in this special case. 

For $t\in\R$, $r_t\circ \alpha$ denote the cycle that results from the operation of composing each of the $n$-cells $\gamma_{i}$ that appear in $\alpha$ with the isotopy $r$:
\[\textrm{if $\alpha=\sum_ic_{i}\gamma_{i}$, $c_i\in\R$, then $r_t\circ\alpha=\sum_ic_{i}\,r_t\circ \gamma_{i}$.}\] 
The variation $\mu_t=r_t^*\mu_\alpha$ constructed above is precisely the same as $\mu_{r_t\circ\alpha}$.

We want to examine what happens when the measure $\mu$ is critical for $A_L$ with respect to all such variations $\mu_t$ for all vector fields $X$. 
We will assume that the support of the derivative of the variation $d\mu_s/ds|_{s=0}$ is contained within a chart, and we will work in local coordinates.
We will write $dt=dt_1\cdots dt_n$. We denote the partial derivatives of $L$ by $L_x$ and $L_{v_i}$.
For each such variation we have:
\begin{multline*}
0=\left.\frac{d}{ds}\right|_{s=0}\int L\,d\mu_s=
\left.\frac{d}{ds}\right|_{s=0}\sum_ic_i\int L(d(r_s\circ\gamma_i)) \,dt\\
=\sum_ic_i\int \left[L_x(d\gamma_i) \left.\frac{\partial r_s\circ \gamma_i}{\partial s}\right|_{s=0}+\sum_jL_{v_j}(d\gamma_i)\left.\frac{\partial^2( r_s\circ\gamma_i)}{\partial s\,\partial t_j}\right|_{s=0}\right]dt\\
=\sum_i c_i\int\left[L_x(d\gamma_i)-\sum_j \frac{\partial L_{v_j}(d\gamma_i)}{\partial t_j}\right]\left.\frac{\partial (r_s\circ\gamma_i)}{\partial s}
\right|_{s=0}dt\\=\sum_ic_i\int\textrm{(E-L)}\left.\frac{\partial r_s}{\partial s}\right|_{s=0}dt
\end{multline*}
where
\[
\textrm{(E-L)}\coloneqq\frac{\partial L}{\partial x}-\sum_{i=1}^n\left(
\frac{\partial^2L}{\partial x\partial v_i}v_i
+\sum_{j=1}^n\frac{\partial^2L}{\partial v_i\partial v_j}\mathsf X_{ij}
\right),
\] 
and a point in the vector space $T_{(v_1,\dots,v_n)}(T_x^nM)$ has coordinates $\mathsf X_{ij}$, $1\leq i,j\leq n$. Since the above is true for all smooth vectorfields $X=\partial r_s/\partial s|_{s=0}$, we conclude that (E-L) must vanish identically throughout the support of $\mu=\mu_\alpha$.

In other words, Condition \critcondition\ for measures $\mu_\alpha$ and for distributions of the form \eqref{eq:horizontaldistribution} is equivalent to the Euler-Lagrange equations (E-L).

\begin{rmk}
In the case of an arbitrary holonomic measure (not necessarily induced by a cycle) we have no information about the second derivatives, so we find no clear way to give this deduction in that general case. While it can be ascertained that these equations must be respected in a weak sense (if $\mu=\lim_i\mu_{\alpha_i}$, the measures $\mu_{\alpha_i}$ will asymptotically satisfy Euler-Lagrange in the sense of distributions, so (E-L) must vanish $\mu$-almost everywhere), it is not clear to us how this can be useful.\qed
\end{rmk}

%% file: vertical.tex
Let $\mu$ be a holonomic measure in $T^nM$.

Recall that $g$ is the Riemannian metric on $M$, and for 
\[u=(u_1,\dots,u_n),\,v=(v_1,\dots,v_n)\in T^nM\] 
let
\[g(u,v)=\sum_{i=1}^ng(u_i,v_i).\]
The set of gradients $\nabla_v d\omega$ of exact differential forms (viewed as functions 
on $T^nM$) is a subspace $F$ of $\mathcal H$. The gradient $\nabla_{v_i} f$ of a function on $f$ on $T^nM$ is defined by
\[g(u_i,\nabla_{v_i}f)=\lim_{t\to0}\frac{f(x,v_1,\dots,v_i+tu_i,\dots,v_n)-f(x,v_1,\dots,v_n)}{t}\]
for all vector fields $u_i$ on $M$,
and $\nabla_vf$ is defined by
\[g((u_1,\dots,u_n),\nabla_v f)=\sum_i g(u_i,\nabla_{v_i}f).\]

We introduce the Hilbert space $\mathcal H$ of all functions 
\[u:\supp\mu\subseteq T^nM\to T^nM\] 
such that $u(x,v)\in T^n_xM$ for all $(x,v)\in T^nM$, and $\int g(u,u)\,d\mu<+\infty$, where $g$ is the Riemannian metric on $M$. The inner product in $\mathcal H$ is defined by
\[(u,u')=\int g(u,u')\,d\mu.\]

Each function $u$ in $\mathcal H$ induces a distribution $\eta^u$ of the form
\[\langle \eta^u,f\rangle=\int_{T^nM}g(u,\nabla_v f)\,d\mu.\]
for $f\in\Cinfc$. This distribution clearly satisfies Conditions \diffcondition\ and \probcondition. The set of all functions $u$ in $\mathcal H$ such that $\eta^u$ satisfies Condition \holcondition\ as well are exactly the orthogonal complement $F^\perp$ to $F$ in $\mathcal H$ because Condition \holcondition\ is
\[0=\langle\eta^u,d\omega\rangle=\int g(u,\nabla_vd\omega)\,d\mu=(u,d\omega).\]
for $\omega\in\Omega^{n-1}(M)$.

It follows that, if Condition \critcondition\ is satisfied for all $\eta^u$ satisfying Conditions \diffcondition, \holcondition, and \probcondition, then  $\nabla_vL$ must be contained in the space $F^{\perp\perp}$, which coincides with the topological closure $\overline F$. We have proved

\begin{prop}[``$L_v=d\omega$'']\label{prop:domega}
If $\mu$ is a holonomic measure that is critical for $A_L$, then there exist a sequence $\{\omega^i\}_i\subset \Omega^{n-1}(M)$ such that
\[\nabla_vL|_{\supp\mu}=\lim_{i\to\infty}\nabla_vd\omega^i.\]
The limit is taken in $\mathcal H$.\qed
\end{prop}


It is possible to produce an explicit variation $\mu_t^u$ of $\mu$ with derivative $\eta^u$ by letting
\[\int_{T^nM} f\,d\mu_s^u=\int_{T^nM} f(x,v+su(x,v))\,d\mu(x,v)\]
for all $f\in \Cinfc$ and $s\in \R$.
It follows from the construction that this variation preserves homology whenever Condition \homcondition\ holds. That is, whenever $u$ is such that
\[0=\langle\eta^u,\omega\rangle=(u,\omega)\]
for all closed forms $\omega\in \Omega^n(M)$. Hence, the same argument as before yields
\begin{prop}
 If $\mu$ is a holonomic measure that is critical for $A_L$ within its homology class, then there exists a sequence of closed $n$-forms $\{\omega^i\}_i\subset \Omega^{n}(M)$, $d\omega^i=0$, such that
 \[\nabla_vL|_{\supp\mu}=\lim_{i\to\infty}\nabla_v \omega^i.\]
 The limit is taken in $\mathcal H$.\qed
\end{prop}

%% file: transpositional.tex
Let $\mu\in\holonomic$ again be a holonomic measure, and let $L$ be a Lagrangian. We will define a type of variation that constitutes for holonomic measures the analogue of a reparameterization, and we will deduce a version of Noether's theorem.

Let $\sigma\in \Cinfc$ and fix some $1\leq i\leq n$.
We consider the distribution on $T^nM$ given by
\[\langle\eta,f\rangle=\int \sigma f\,d\mu-\int \sigma \,g(v_i,\nabla_{v_i}f)\,d\mu-\int f\,d\mu\int \sigma\, d\mu\]
for $f\in\Cinfc$. The distribution $\eta$ clearly satisfies Conditions \diffcondition\ and \probcondition. To see that it also satisfies Condition \holcondition, we compute, for $\omega\in\Omega^{n-1}(M)$,
\[\langle\eta,d\omega\rangle=\int \sigma \,d\omega\,d\mu-\int \sigma \,d\omega\,d\mu-\int d\omega\,d\mu\int\sigma\,d\mu=0.\]
Here, we used that $g(v_i,\nabla_{v_i}d\omega)=d\omega$ by linearity, and we also used the fact that $\mu$ is holonomic.

If $\mu$ is critical for $A_L$, it must satisfy Condition \critcondition\ for all variations arising in this way from any $\sigma\in \Cinfc$.
This translates to
\[0=\langle\eta,L\rangle=\int\sigma L\,d\mu-\int\sigma\, g(v_i,\nabla_{v_i}L)\,d\mu-\int L\,d\mu\int\sigma\,d\mu.\]
If the domain of $\sigma$ is very small around a point $(x,v)\in T^nM$, this can be very well approximated by
\[0\approx \int\sigma\,d\mu\left(L(x,v)-g(v_i,\nabla_{v_i}L(x,v))-\int L\,d\mu\right).\]
This is how we deduce
\begin{prop}[Energy conservation]\label{prop:hamiltonian}
If a holonomic measure is critical with respect to all transpositional variations, then its support is a subset of the set where
\[g(v_i, \nabla_{v_i}L)-L=-A_L(\mu). \]
\qed
\end{prop}
\begin{rmk}\label{rmk:energyconservation}
In the cases in which we can define the change of variables $p_i=L_{v_i}$ (for example, in the case of convex, superlinear Lagrangians), we can also define the Hamiltonians
\[H_i(x,p_i)=H_i(x,p_i;v_1,\dots,\widehat v_i,\dots,v_n)=
p_iv_i-L(x,v_1,\dots,v_n),\]
and what we have here is just a higher-dimensional version of the usual energy conservation principle.\qed
\end{rmk}
\begin{rmk}[Hamilton-Jacobi equation]\label{rmk:hamiltonjacobi}
It follows from Propositions \ref{prop:domega} and \ref{prop:hamiltonian} that $g(v_i,\nabla_{v_i}L)$ is independent of $i$ and equal to $\lim_kd\omega^k$ on $\supp\mu$. Hence also
\[\lim_{k\to\infty}H_i(x,d\omega^k)=-A_L(\mu)\]
on $\supp\mu$, for all $i$.
This is a generalized form of the Hamilton-Jacobi equation. This situation is very similar to the weak KAM theorem of Fathi \cite{fathibook}, with the important caveat that we have proved nothing about the regularity of the limit of the forms $\omega^k$. \qed
\end{rmk}

The distribution $\eta$ is in fact the derivative of the variation $\mu_t^\sigma$ given by
\begin{equation*}
\int f\,d\mu_t^\sigma=-\frac{\int (1-t\sigma)f(x,v_1,\dots,v_{i-1},(1-t\sigma)^{-1}v_i,v_{i+1},\dots,v_n)\,d\mu} {\int(1-t\sigma)\,d\mu}
\end{equation*}
for $f\in \Cinfc$ and for $t$ in an open interval that contains 0.

If we require the variation $\mu^\sigma_t$ to preserve homology, then we find that we must require $\int\sigma\,d\mu=0$ because
\[\int\omega\,d\mu_t^\sigma=\frac{\langle\rho(\mu),\omega\rangle}{\int (1-t\sigma)\,d\mu}\]
must be constant for each closed form $\omega\in \Omega^{n}(M)$, $d\omega=0$.
It follows that if $\mu$ is critical for $A_L$ within its homology class then it must satisfy
\[\int \sigma(L-L_{v_i}\cdot v_i)\,d\mu=0\]
for all $\sigma$ with $\int\sigma\,d\mu=0$. We can use $\sigma d\mu$ to approximate the derivative at any point in $\supp\mu$ arbitrarily well. Hence, we get
\begin{prop}[Energy conservation for homological minimizers]
If a holonomic measure is critical for $A_L$ within its homology class, then for each connected component $K$ of $\supp\mu$ there are some $c_1,\dots,c_n\in \R$ such that $K$ is contained in the set where
\[L-g(v_i,\nabla_{v_i} L)=c_i,\quad i=1,2,\dots,n.\]
\qedhere
\end{prop}

%% file: hol3.bbl
\begin{thebibliography}{10}

\bibitem{allard}
William~K. Allard.
\newblock On the first variation of a varifold.
\newblock {\em Ann. of Math. (2)}, 95:417--491, 1972.

\bibitem{almgrenvarifolds}
Frederick~J. Almgren, Jr.
\newblock {\em Plateau's problem: {A}n invitation to varifold geometry}.
\newblock W. A. Benjamin, Inc., New York-Amsterdam, 1966.

\bibitem{ambrosio}
Luigi Ambrosio, Nicola Gigli, and Giuseppe Savar{\'e}.
\newblock {\em Gradient flows in metric spaces and in the space of probability
  measures}.
\newblock Lectures in Mathematics ETH Z\"urich. Birkh\"auser Verlag, Basel,
  second edition, 2008.

\bibitem{ambrosiogiglisavare}
Luigi Ambrosio, Nicola Gigli, and Giuseppe Savar{\'e}.
\newblock {\em Gradient flows in metric spaces and in the space of probability
  measures}.
\newblock Lectures in Mathematics ETH Z\"urich. Birkh\"auser Verlag, Basel,
  second edition, 2008.

\bibitem{bangert}
V.~Bangert.
\newblock Minimal measures and minimizing closed normal one-currents.
\newblock {\em Geom. Funct. Anal.}, 9(3):413--427, 1999.

\bibitem{patrick}
Patrick Bernard.
\newblock Young measures, superposition and transport.
\newblock {\em Indiana Univ. Math. J.}, 57(1):247--275, 2008.

\bibitem{colombeau}
Jean-Fran{\c{c}}ois Colombeau.
\newblock {\em New generalized functions and multiplication of distributions},
  volume~84 of {\em North-Holland Mathematics Studies}.
\newblock North-Holland Publishing Co., Amsterdam, 1984.
\newblock Notas de Matem{\'a}tica [Mathematical Notes], 90.

\bibitem{colombeaubook}
Jean-Fran{\c{c}}ois Colombeau.
\newblock {\em Multiplication of distributions}, volume 1532 of {\em Lecture
  Notes in Mathematics}.
\newblock Springer-Verlag, Berlin, 1992.
\newblock A tool in mathematics, numerical engineering and theoretical physics.

\bibitem{contrerasiturriagabook}
Gonzalo Contreras and Renato Iturriaga.
\newblock {\em Global minimizers of autonomous {L}agrangians}.
\newblock 22$^{\rm o}$ Col\'oquio Brasileiro de Matem\'atica. [22nd Brazilian
  Mathematics Colloquium]. Instituto de Matem\'atica Pura e Aplicada (IMPA),
  Rio de Janeiro, 1999.

\bibitem{crandalllions}
Michael~G. Crandall and Pierre-Louis Lions.
\newblock Viscosity solutions of {H}amilton-{J}acobi equations.
\newblock {\em Trans. Amer. Math. Soc.}, 277(1):1--42, 1983.

\bibitem{dacorogna}
Bernard Dacorogna.
\newblock {\em Introduction to the calculus of variations}.
\newblock Imperial College Press, London, second edition, 2009.
\newblock Translated from the 1992 French original.

\bibitem{david14}
Guy David.
\newblock Should we solve {P}lateau's problem again?
\newblock In C.~Fefferman, A.~D. Ionescu, D.~H. Phong, and S.~Wainger, editors,
  {\em Advances in {A}nalysis: the legacy of {E}lias {M}. {S}tein}. Princeton
  University Press, 2014.

\bibitem{delellisghiraldinmaggi}
Camillo De~Lellis, Francesco Ghiraldin, and Francesco Maggi.
\newblock A direct approach to plateau's problem.

\bibitem{depauw09}
Thierry De~Pauw.
\newblock Size minimizing surfaces.
\newblock {\em Ann. Sci. \'Ec. Norm. Sup\'er. (4)}, 42(1):37--101, 2009.

\bibitem{dephillipsderosaghiraldin}
Guido De~Philippis, Antonio De~Rosa, and Francesco Ghiraldin.
\newblock A direct approach to {P}lateau's problem in any codimension.

\bibitem{eellssampson}
James Eells, Jr. and J.~H. Sampson.
\newblock Harmonic mappings of {R}iemannian manifolds.
\newblock {\em Amer. J. Math.}, 86:109--160, 1964.

\bibitem{evans}
Lawrence~C. Evans.
\newblock Quasiconvexity and partial regularity in the calculus of variations.
\newblock {\em Arch. Rational Mech. Anal.}, 95(3):227--252, 1986.

\bibitem{fang13}
Yangqin Fang.
\newblock Existence of minimizers for the {R}eifenberg {P}lateau problem.

\bibitem{fathibook}
Albert Fathi.
\newblock Weak {KAM} theorem in lagrangian dynamics.
\newblock Preliminary Version Number 10, June 2008.

\bibitem{federer}
Herbert Federer.
\newblock {\em Geometric measure theory}.
\newblock Die Grundlehren der mathematischen Wissenschaften, Band 153.
  Springer-Verlag New York Inc., New York, 1969.

\bibitem{feuvrier09}
Vincent Feuvrier.
\newblock Condensation of polyhedric structures onto soap films.

\bibitem{friedlander}
F.~G. Friedlander.
\newblock {\em Introduction to the theory of distributions}.
\newblock Cambridge University Press, Cambridge, second edition, 1998.
\newblock With additional material by M. Joshi.

\bibitem{giusti}
Enrico Giusti.
\newblock {\em Direct methods in the calculus of variations}.
\newblock World Scientific Publishing Co., Inc., River Edge, NJ, 2003.

\bibitem{harrison14}
Jenny Harrison.
\newblock Operator calculus of differential chains and differential forms.

\bibitem{harrison11}
Jenny Harrison.
\newblock Soap film solutions to {P}lateau's problem.

\bibitem{harrisonpugh}
Jenny Harrison and Harrison Pugh.
\newblock Existence and soap film regularity of solutions to {P}lateau's
  problem.

\bibitem{jost}
J{\"u}rgen Jost.
\newblock {\em Riemannian geometry and geometric analysis}.
\newblock Universitext. Springer, Heidelberg, sixth edition, 2011.

\bibitem{kristensenmingione}
Jan Kristensen and Giuseppe Mingione.
\newblock The singular set of minima of integral functionals.
\newblock {\em Arch. Ration. Mech. Anal.}, 180(3):331--398, 2006.

\bibitem{leonardimasnou}
Gian~Paolo Leonardi and Simon Masnou.
\newblock Locality of the mean curvature of rectifiable varifolds.
\newblock {\em Adv. Calc. Var.}, 2(1):17--42, 2009.

\bibitem{maggi}
Francesco Maggi.
\newblock {\em Sets of finite perimeter and geometric variational problems},
  volume 135 of {\em Cambridge Studies in Advanced Mathematics}.
\newblock Cambridge University Press, Cambridge, 2012.
\newblock An introduction to geometric measure theory.

\bibitem{manhe}
Ricardo Ma{\~n}{\'e}.
\newblock {\em Ergodic theory and differentiable dynamics}, volume~8 of {\em
  Ergebnisse der Mathematik und ihrer Grenzgebiete (3) [Results in Mathematics
  and Related Areas (3)]}.
\newblock Springer-Verlag, Berlin, 1987.
\newblock Translated from the Portuguese by Silvio Levy.

\bibitem{matheractionminimizing91}
John~N. Mather.
\newblock Action minimizing invariant measures for positive definite
  {L}agrangian systems.
\newblock {\em Math. Z.}, 207(2):169--207, 1991.

\bibitem{morgan}
Frank Morgan.
\newblock {\em Geometric measure theory}.
\newblock Elsevier/Academic Press, Amsterdam, fourth edition, 2009.
\newblock A beginner's guide.

\bibitem{moser1986}
J{\"u}rgen Moser.
\newblock Minimal solutions of variational problems on a torus.
\newblock {\em Ann. Inst. H. Poincar\'e Anal. Non Lin\'eaire}, 3(3):229--272,
  1986.

\bibitem{moser1988}
J{\"u}rgen Moser.
\newblock A stability theorem for minimal foliations on a torus.
\newblock {\em Ergodic Theory Dynam. Systems}, 8$^*$(Charles Conley Memorial
  Issue):251--281, 1988.

\bibitem{moser1987}
J{\"u}rgen Moser.
\newblock Minimal foliations on a torus.
\newblock In {\em Topics in calculus of variations ({M}ontecatini {T}erme,
  1987)}, volume 1365 of {\em Lecture Notes in Math.}, pages 62--99. Springer,
  Berlin, 1989.

\bibitem{reifenberg60}
E.~R. Reifenberg.
\newblock Solution of the {P}lateau {P}roblem for {$m$}-dimensional surfaces of
  varying topological type.
\newblock {\em Acta Math.}, 104:1--92, 1960.

\bibitem{reifenberg64a}
E.~R. Reifenberg.
\newblock An epiperimetric inequality related to the analyticity of minimal
  surfaces.
\newblock {\em Ann. of Math. (2)}, 80:1--14, 1964.

\bibitem{reifenberg64b}
E.~R. Reifenberg.
\newblock On the analyticity of minimal surfaces.
\newblock {\em Ann. of Math. (2)}, 80:15--21, 1964.

\bibitem{myclosedmeasuresareholonomic}
Rodolfo R\'ios-Zertuche.
\newblock Polygonal approximations of closed parametric varifolds.
\newblock Preprint. arXiv:1409.1205 [math.AP].

\bibitem{simon}
Leon Simon.
\newblock {\em Lectures on geometric measure theory}, volume~3 of {\em
  Proceedings of the Centre for Mathematical Analysis, Australian National
  University}.
\newblock Australian National University Centre for Mathematical Analysis,
  Canberra, 1983.

\bibitem{smolyanov}
O.~G. Smolyanov and H.~von Weizs{\"a}cker.
\newblock Differentiable families of measures.
\newblock {\em J. Funct. Anal.}, 118(2):454--476, 1993.

\bibitem{young}
L.~C. Young.
\newblock {\em Lectures on the calculus of variations and optimal control
  theory}.
\newblock Foreword by Wendell H. Fleming. W. B. Saunders Co., Philadelphia,
  1969.

\end{thebibliography}
